%% file: multiphase_part1_18.tex
\documentclass[11pt]{article}
\usepackage{amsfonts,amsmath,amsthm,amssymb}
\usepackage{wrapfig}
\usepackage{graphicx}
\usepackage[colorlinks,pdfpagelabels,pdfstartview = FitH,bookmarksopen
= true,bookmarksnumbered = true,linkcolor = blue,plainpages =
false,hypertexnames = false,citecolor = red,pagebackref=false]{hyperref}
\usepackage{relsize}
\usepackage{color}
\usepackage{appendix}
\usepackage{ulem}
\usepackage[makeroom]{cancel}
\usepackage[margin=1.4in]{geometry}
\usepackage{stmaryrd}

\allowdisplaybreaks

\let\TeXchi\chi
\newbox\chibox
\setbox0 \hbox{\mathsurround0pt $\TeXchi$}
\setbox\chibox \hbox{\raise\dp0 \box 0 }
\def\chi{\copy\chibox}

\renewcommand{\d}{\mathrm{d}}
\newcommand{\dx}{\mathrm{d}x}

\newcommand{\dt}{\mathrm{d}t}
\newcommand{\ds}{\mathrm{d}s}
\renewcommand{\rho}{\varrho}

\newcommand{\umenkp}{(u-k)_+}
\newcommand{\umenpm}{(u-k)_{\pm}}

\input dibe_1009.tex

\begin{document}
\title{Continuity of the temperature in a multi-phase transition problem }
\author{Ugo Gianazza\\
Dipartimento di Matematica ``F. Casorati", 
Universit\`a di Pavia\\ 
via Ferrata 5, 27100 Pavia, Italy\\
email: {\tt gianazza@imati.cnr.it}
\and Naian Liao\\
Fachbereich Mathematik, Universit\"at Salzburg\\
Hellbrunner Str. 34, 5020 Salzburg, Austria\\
email: {\tt naian.liao@sbg.ac.at}
}
\date{}
\maketitle
\begin{abstract}
Locally bounded, local weak solutions to a doubly nonlinear parabolic equation,
which models the multi-phase transition of a material, is shown to be locally continuous.
Moreover, an explicit modulus of continuity is given. The effect  of the $p$-Laplacian type diffusion 
is also considered.
\end{abstract}
\section{Introduction}
The temperature $u$ of a material undergoing a multi-phase change,
for instance  ice-water-vapor, can be described by the following
nonlinear parabolic partial differential equation
\begin{equation}\label{Eq:1:0}
\pl_t\be(u)-\dvg\big(|Du|^{p-2}Du\big)\ni0\quad\text{ weakly in }E_T. 
\end{equation}
Here $E$ is an open set of $\rn$ with $N\ge1$ and  $E_T:=E\times(0,T]$ for some $T>0$.
The enthalpy $\be(\cdot)$ is a maximal monotone graph in $\rr\times\rr$ defined by (cf.~Figure~\ref{Fig-1})
\begin{equation}\label{Eq:beta}
\be(u)=u+\sum_{i=0}^{\ell} \nu_i \mathcal{H}_{e_i}(u)\quad\text{ for some } \ell\in\nn\cup\{\infty\}, \, e_i\in\rr \text{ and } \nu_i>0,
\end{equation}
where we have assumed that $0=e_o<e_1<\dots<e_{\ell}$, 
$$d:=\min\big\{e_{i+1}-e_{i}: i=0,1,\cdots,\ell-1\big\}>0,$$
and  denoted
\begin{equation*}
\mathcal{H}_{e_i}(u)=\left\{
\begin{array}{cc}
1,\quad& u>e_i,\\[5pt] 
[0,1],\quad& u=e_i,\\[5pt]
0,\quad& u<e_i.
\end{array}
\right.
\end{equation*} 
The equation \eqref{Eq:1:0} will be understood in a proper weak sense
 to be made precise later.

The {\bf main result} is that locally bounded, local weak solutions to \eqref{Eq:1:0} with $p\ge2$
are locally continuous and a modulus of continuity is explicitly quantified.

\begin{figure}[t]\label{Fig-1}
\centering
\includegraphics[scale=1]{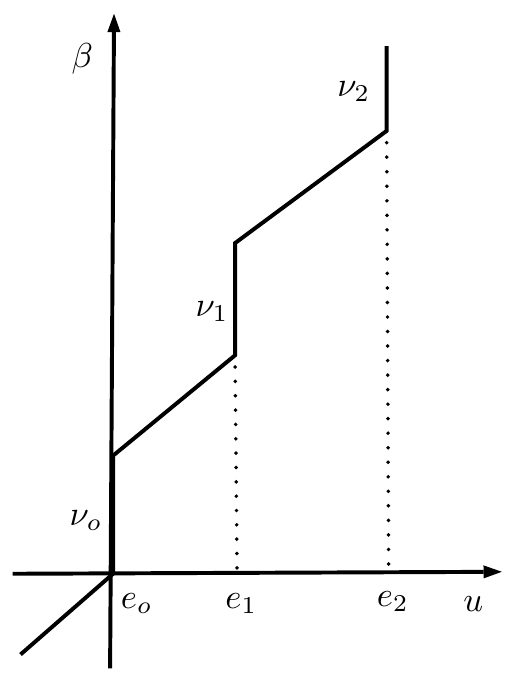}
\caption{Graph of $\be$}
\end{figure}
\subsection{Statement of the results}
From here on, we will deal with the following
more general parabolic partial differential equation modeled on \eqref{Eq:1:0}:
\begin{equation}\label{Eq:1:1}
\pl_t\be(u)-\dvg\bl{A}(x,t,u, Du) \ni 0\quad \text{ weakly in }\> E_T,
\end{equation}
where $\be(\cdot)$ is defined in \eqref{Eq:beta}.   
The function $\bl{A}(x,t,u,\xi)\colon E_T\times\rr^{N+1}\to\rn$ is assumed to be
measurable with respect to $(x, t) \in E_T$ for all $(u,\xi)\in \rr\times\rn$,
and continuous with respect to $(u,\xi)$ for a.e.~$(x,t)\in E_T$.
Moreover, we assume the structure conditions
\begin{equation}  \label{Eq:1:2}
\left\{
\begin{array}{l}
\bl{A}(x,t,u,\xi)\cdot \xi\ge C_o|\xi|^p \\[5pt]
|\bl{A}(x,t,u,\xi)|\le C_1|\xi|^{p-1}%
\end{array}%
\right .\quad \text{ a.e.}\> (x,t)\in E_T,\, \forall\,u\in\rr,\,\forall\,\xi\in\rn,
\end{equation}
where $C_o$ and $C_1$ are given positive constants, and we take $p\ge2$.

In the sequel, the set of parameters $\{d, \nu_i, p,N,C_o,C_1, \|u\|_{\infty, E_T}\}$ will be referred to as the   data.
A generic positive constant $\gm$ depending on the data will be used in the estimates.

Let $\Gm:=\pl E_T-\overline{E}\times\{T\}$
be the parabolic boundary of $E_T$, and for a compact set $\mathcal{K}\subset E_T$
introduce the parabolic $p$-distance from $\mathcal{K}$ to $\Gm$ by
\begin{equation*}
	\begin{aligned}
		\dist_p(\mathcal{K};\,\Gm)&:=\inf_{\substack{(x,t)\in \mathcal{K}\\(y,s)\in\Gm}}
		\left\{|x-y|+|t-s|^{\frac1p}\right\}.
	\end{aligned}
\end{equation*}

The formal definition of local weak solution to \eqref{Eq:1:1} will be given in \S~\ref{S:1:3}.
Now we proceed to present the main theorem, where by $\ln^{(k)}$ we mean the logarithmic function composed $k$ times.
\begin{theorem}\label{Thm:1:1}
Let $u$ be a bounded weak solution to \eqref{Eq:1:1} in $E_T$, 
 under the structure condition \eqref{Eq:1:2} for $p\ge2$.  
Then  
for every pair of points $(x_1,t_1), (x_2,t_2)\in \mathcal{K}$,
there holds that   
	\begin{equation*}
	\big|u(x_1,t_1)-u(x_2,t_2)\big|
	\le
	\boldsymbol\om\!\left(|x_1-x_2|+|t_1-t_2|^{\frac1p}\right),
	\end{equation*}
where 
$$\boldsymbol\om(r)=C \Big(\ln^{(6)} \frac{\dist_p(\mathcal{K};\,\Gm)}{c r}\Big)^{-\sig}\quad\text{ for all }r\in\big(0, \dist_p(\mathcal{K};\,\Gm)\big)$$ 
for some absolute constant $c\in(0,1)$, and for some $C>1$ and $\sig\in(0,1)$ depending on the data.
\end{theorem}
\begin{remark}\upshape
All constants in Theorem~\ref{Thm:1:1} are stable as $p\downarrow2$. 
\end{remark}
\begin{remark}\upshape
Even though all the proofs are given for the specific $\beta$  in \eqref{Eq:beta}, nevertheless a more general graph can be considered, namely 
\begin{equation}\label{Eq:gen:beta}
\be(u)=\be_{AC}(u)+\sum_{i=0}^{\ell} \nu_i \mathcal{H}_{e_i}(u)\quad\text{ for some } \ell\in\nn\cup\{\infty\},\, e_i\in\rr \text{ and } \nu_i>0,
\end{equation}
where $\beta_{AC}=\be_{AC}(s)$ denotes an absolutely continuous and hence a.e. differentiable function in $\rr$, such that 
\[
0<\al_o\leq\beta_{AC}'(s)\leq\al_1,
\]
for two positive constants $\alpha_o$ and $\alpha_1$. This reflects the fact that the thermal properties of the material under consideration might change according to the temperature.
The graph \eqref{Eq:gen:beta} can be reduced to \eqref{Eq:beta} by a straightforward adaption of the change of variables introduced in \cite[\S~1]{DB-82}. Furthermore, Theorem~\ref{Thm:1:1} continues to hold for \eqref{Eq:1:1} with lower order terms, which take into account the convection resulting from the heat transfer. Again, the modifications of the proofs can be modeled on the arguments in \cite{DB-82,DB-86,DB}, but we refrain from pursuing generality in this direction, focusing instead  on the actual novelties.
\end{remark}

Theorem~\ref{Thm:1:1} bears global information of $\be$ through the range of $u$. 
However, once a modulus of continuity is obtained,
we can confine the range of $u$ by restricting space-time distance, such that $u$ only experiences one jump of $\be$ at most.
\begin{corollary}[Localization]\label{Cor:1:2} 
Under the hypotheses of Theorem~\ref{Thm:1:1},
the modulus improves automatically to the one for the two-phase problem.
\end{corollary}
\subsection{Novelty and significance}
Graphs $\be$ such as the one in \eqref{Eq:beta}, but exhibiting just a single jump, say at the origin, arise from a weak formulation of the classical Stefan problem, which models a liquid/solid phase transition, such as water/ice. It is quite natural to ask whether the transition of phase occurs with a continuous temperature across the water/ice interface. This question was initially raised in a 1960 paper of {Ole\u{\i}nik} (see \cite{Olienik}) and was later reported in \cite[Chapter~V, \S~9]{LSU}. Since then an important research field was born, and soon new problems  started to be posed, besides the one originally formulated by {Ole\u{\i}nik} in her 1960 paper. The interested reader can refer to \cite{Visintin}, to have at least an overview of the huge development that the research about the Stefan problem has witnessed. In these notes the issue is the regularity of local solutions, the ultimate goal being to prove the continuity of solutions to \eqref{Eq:1:0} for a general maximal monotone graph $\be$. Such a result has not been achieved yet, even though it is clear that the coercivity of $\be$ is essential for a solution to be continuous, as pointed out by examples  in \cite{DiBe-Gar}.

Continuity results for \eqref{Eq:1:0} with $\be$ as in \eqref{Eq:beta} but with a {\it single} jump, and $p=2$, have been given in \cite{Caff-Evans-83, DB-82,Sacks,Ziemer}. 
Moreover, Ziemer proved the continuity up to the boundary for general Dirichlet boundary data. Whereas Caffarelli and Evans heavily relied on the properties of the Laplacian, and their result cannot be extended to the full quasilinear case of \eqref{Eq:1:1}, 
DiBenedetto's approach is flexible enough to deal with the general framework, and it also allows lower order terms, which are thoroughly justified from a physical point of view, since they describe convection phenomena.

A quantitative estimate on the modulus of continuity, still in the case of a single jump and $p=2$, was given in \cite[Remark~3.1]{DiBe-Fried}, but without proof. Few years later, DiBenedetto quantified Ziemer's results, and in \cite{DB-86} proved that solutions have a boundary modulus of continuity of the kind
\begin{equation}\label{Eq:lnln}
\boldsymbol\om(r)=C\Big[\ln\ln\Big(\frac{R}{cr}\Big)\Big]^{-\sig},\qquad C>1,\, c,\,\sig\in(0,1),\,r\in(0,R).
\end{equation}

A major step forward towards a full proof of the local continuity of solutions to \eqref{Eq:1:1} with $p=2$ and $\be$ a general maximal monotone graph in $\rr\times\rr$, is represented by \cite{DBV-95}; the authors proved that locally bounded, weak solutions are locally continuous, and the modulus of continuity can be quantitatively estimated only in terms of the data, even though an explicit expression of such a modulus is not provided in the paper. The proof is given in full generality for $N=2$, whereas for $N\ge3$ it relies on a proper comparison function, and therefore, it is limited to ${\bl A}=Du$. The paper is quite technical, but a thorough and clear presentation of the methods employed is given in \cite[\S~5]{DBUV}; the list of references therein gives a comprehensive state of the art at the moment of its publication.

To our knowledge, the first paper to deal with $p>2$ is \cite{Urbano-00}: besides its intrinsic mathematical interest, the nonlinear diffusion operator with growth of order larger than 2 naturally takes into account non-Newtonian filtration phenomena.

For a few years there were basically no further improvements, as far as the continuity issue is concerned. Things changed with \cite{Urbano-14}: the authors consider \eqref{Eq:1:1} with $p\ge2$ and \eqref{Eq:beta} with a single jump, and they derive an explicit modulus of continuity better than \eqref{Eq:lnln}, namely
\begin{equation}\label{Eq:ln}
\boldsymbol\om(r)=C\Big[\ln\Big(\frac{R}{cr}\Big)\Big]^{-\sig},\qquad C>1,\,  c,\,\sig\in(0,1),\,r\in(0,R),
\end{equation}
with $\sig$ precisely quantified just in terms of $N$ and $p$, which they conjecture to be optimal. In \cite{BKLU-18} the result is extended up to the boundary: under the same conditions as before about the equation, and assuming a positive geometric density condition at the boundary $\partial E$, solutions to the Dirichlet problem have a modulus of continuity as in \eqref{Eq:lnln}, yet weaker than \eqref{Eq:ln}.

Further progress has been recently made in \cite{Liao-Stefan, Liao-improved}.
Indeed, interior moduli  sharper than \eqref{Eq:ln} are provided in \cite{Liao-improved} for $p=2$ and $N=1,2$.
On the other hand, under the same general conditions as in \cite{BKLU-18}, 
the boundary modulus of continuity has been  improved to \eqref{Eq:ln} in \cite{Liao-Stefan}: for Dirichlet boundary conditions, any $p\ge2$ can do; whereas for Neumann boundary conditions only $p=2$ could be dealt with, while the case $p>2$ remains an open problem.

With respect to the existing literature described so far, the present work represents a step forward, at least under two different points of view.

First of all, we consider an arbitrary number of jumps of $\be$, and not just a single discontinuity; this case has already been dealt with in \cite{GV-03}, but only for $p=2$, whereas here we work with $p\ge2$. Moreover, even though some of the techniques employed in \cite{GV-03} and here are comparable, the general approach we follow is definitely different.

 The other novelty is given by the explicit modulus of continuity in Theorem~\ref{Thm:1:1}:
to our knowledge, it is the first time that a modulus is explicitly stated for a $\be$ that is more general than the one considered in \cite{Urbano-14,BKLU-18,Liao-Stefan, Liao-improved}.  Due to the wide generality assumed on $\be$, i.e. arbitrary number of jumps and arbitrary height for each single jump, the parameter  $\sig$ depends on the data, that is, also on $\|u\|_{\infty,E_T}$.
Providing an optimal modulus of continuity that carries global information of $\be$ is a difficult task, and we are well aware that the one shown in Theorem~\ref{Thm:1:1} seems far from being the best possible. Nevertheless,  as we have  pointed out in Corollary~\ref{Cor:1:2}, the importance of a {\it quantitative} continuity statement lies in the fact that once we have it, the same result implies that the modulus can be automatically improved to the one
for the two-phase problem (single-jump); indeed, by restricting the space-time
distance, $u$ can be confined, so that it experiences one jump of $\be$ at most, and we end up having the modulus given in \eqref{Eq:ln}. We refrained from going into details about the proof of Corollary~\ref{Cor:1:2}, since we would basically have to reproduce what was done in \cite{Urbano-14}.

%
Moreover, in a forthcoming paper we plan to address a multi-phase transition problem with a maximal monotone graph $\be$ as in \eqref{Eq:gen:beta}, without assuming that  $\be_{AC}'$ is bounded above: besides its intrinsic mathematical interest, this is what occurs, for example, in the so-called Buckley-Leverett model for the motion of two immiscible fluids in a porous medium (see \cite{KrS,Le}). In such a case, $\be$ presents two singularities, say at $u=0$ and $u=1$, where $\be$ can become vertical with an exponential speed, or even faster, and might also exhibit a jump. 
\subsection{Definition of solution}\label{S:1:3} 
A function
\begin{equation*}
	u\in L_{\loc}^{\infty}\big(0,T;L^2_{\loc}(E)\big)\cap L^p_{\loc}\big(0,T; W^{1,p}_{\loc}(E)\big)
\end{equation*}
is a local, weak sub(super)-solution to \eqref{Eq:1:1} with the structure
conditions \eqref{Eq:1:2}, if for every compact set $K\subset E$ and every sub-interval
$[t_1,t_2]\subset (0,T]$, there is a selection $v\subset\be(u)$, i.e.
\[
\left\{\big(z,v(z)\big): z\in E_T\right\}\subset \left\{\big(z,\be[u(z)]\big): z\in E_T\right\},
\]
 such that
\begin{equation*}
	\int_K v\z \,\dx\bigg|_{t_1}^{t_2}
	+
	\iint_{K\times(t_1,t_2)} \big[-v\pl_t\z+\bl{A}(x,t,u,Du)\cdot D\z\big]\dx\dt
	\le(\ge)0
\end{equation*}
for all non-negative test functions
\begin{equation*}
\z\in W^{1,2}_{\loc}\big(0,T;L^2(K)\big)\cap L^p_{\loc}\big(0,T;W_o^{1,p}(K)%
\big).
\end{equation*}
Observe that $v\in L^{\infty}_{\loc} \big(0,T;L^2_{\loc}(E)\big)$ and hence  all the integrals 
are well-defined. A function that is both a local, weak sub-solution and a local, weak super-solution is termed a local, weak solution.

We will consider the regularized version of the Stefan problem \eqref{Eq:1:1}.
For a parameter $\varep\in(0,\tfrac12 d)$,
we introduce the function
\begin{equation*}
\mathcal{H}_{e_i}^{\varep}(u):=\left\{
\begin{array}{cc}
1,\quad& u>e_i+\varep,\\[5pt] 
\frac1{\varep}(u-e_i),\quad& e_i\le u\le e_i+\varep,\\[5pt]
0,\quad& u<e_i,
\end{array}
\right.
\end{equation*} 
and define the mollification of $\be$ by
\begin{equation*}
\be_\varep(u)\equiv u+H_\varep(u):=u+\sum_{i=0}^{\ell} \nu_i \mathcal{H}_{e_i}^{\varep}(u);
\end{equation*}
we now deal with
\begin{equation}\label{Eq:reg}
\begin{aligned}
&\pl_t [u+H_\varep(u)] -\dvg\bl{A}(x,t,u, Du) \le (\ge) 0\quad \text{ weakly in }\> E_T.
\end{aligned}
\end{equation}
Sub(super)-solutions to \eqref{Eq:reg} are defined like for the parabolic $p$-Laplacian as in \cite[Chapter~II]{DB}.
Hence, the  solutions to \eqref{Eq:reg} are generally not smooth. 

In this note we assume that local solutions to \eqref{Eq:1:1}
can be approximated by a sequence of solutions to \eqref{Eq:reg} locally uniformly.
This approximating approach parallels the one in \cite{DBV-95}, yet with a more particular $\be_\varep$ 
and the $p$-Laplacian here. The goal is to establish an estimate on the modulus of continuity
for the approximating solutions uniform in $\varep$, which grants the same modulus to the limiting function.

There is yet another notion of solution, which requires a solution to possess time derivative in the Sobolev sense,
cf.~\cite{DB-82, DB-86, Liao-Stefan, Liao-improved}.
Theorem~\ref{Thm:1:1} continues to hold for that kind of notion and the proof calls for minor modifications from the one given here.
The advantage is that an approximating scheme is not needed. However, the preset requirement on
the time derivative is usually too strong to guarantee the continuity of a constructed solution in the existence theory.

%
\subsection{Structure of the proof}
Since the paper is technically involved, we think it better to first discuss the main ideas in an informal way.

Roughly speaking, we follow an approach that is by now  standard when dealing with the continuity of solutions to degenerate parabolic equations: starting from a properly built reference cylinder, we have two alternatives: either we can find a sub-cylinder, such that the set where $u$ is close to its supremum is small, or such a sub-cylinder cannot be determined. 

In the first case, we can show  a reduction of oscillation near the supremum, and this is accomplished in \S~\ref{S:1st-alt}; the second alternative is more difficult and will be taken on in \S\S~\ref{S:2nd-alt} -- \ref{S:reduc-osc-3}, where we prove a reduction of oscillation near the infimum, assuming that such an infimum is actually close to one of the discontinuity points; finally, the case of the infimum being properly far from all the discontinuity points is dealt with in \S~\ref{S:reduc-osc-4}. Indeed, this last possibility is the easiest one, since the equation behaves as though it were the parabolic $p$-Laplacian with $p\ge2$.

All the alternatives are quantified, and the structural dependences of the various constants are carefully traced, and this eventually leads to an estimate of the modulus of continuity in  \S~\ref{S:modulus}.  As pointed out in Corollary~\ref{Cor:1:2}, once established, the modulus improves automatically to the one for the two-phase problem. Indeed, in general, $\om$ could be large,  and $d$ could be small: our argument shrinks $\om$ step by step to an oscillation less than $d$ across all potential jumps, and this quantification is precisely what eventually gives the modulus of \eqref{Eq:ln}.

\

\noi{\it Acknowledgement. } 
U. Gianazza was supported by the grant 2017TEXA3H\_002 ``Gradient flows, Optimal Transport and Metric Measure Structures". N. Liao was supported by the FWF--Project P31956--N32 ``Doubly nonlinear evolution equations". 
We thank the referees for their careful reading and comments.
\section{Preliminary tools}
\subsection{Energy estimates}
Here and in the following, we denote by $K_\rho(x_o)$ the cube of side length $2\rho$ and center $x_o$, with faces parallel to the coordinate planes of $\rr^N$, and for $k\in\rr$ we let the truncations $(u-k)_+$ and $(u-k)_-$
be defined by
\[
(u-k)_+\equiv\max\{u-k, 0\},\qquad (u-k)_-\equiv\max\{k-u, 0\}.
\] 
We can repeat almost verbatim the calculations in \cite[\S~2]{DBV-95} modulo 
proper mollification in the time variable, and prove the following estimates.
\begin{proposition}\label{Prop:2:1}
	Let $u$ be a  local weak sub(super)-solution to \eqref{Eq:reg} with \eqref{Eq:1:2} in $E_T$.
	There exists a constant $\gm (C_o,C_1,p)>0$, such that
 	for all cylinders $Q_{R,S}=K_R(x_o)\times (t_o-S,t_o)\subset E_T$,
 	every $k\in\rr$, and every non-negative, piecewise smooth cutoff function
 	$\z$ vanishing on $\pl K_{R}(x_o)\times(t_o-S,t_o)$,  there holds
\begin{align}\label{Eq:en-est-gen}
	\essup_{t_o-S<t<t_o}&\frac12\int_{K_R(x_o)\times\{t\}}	
	\z^p (u-k)_\pm^2\,\dx\nonumber\\
	&\quad + \essup_{t_o-S<t<t_o}\int_{K_R(x_o)\times\{t\}}	
	\left(\int_0^{(u-k)_\pm}H'_\varep(k\pm s) s\,\ds\right)\z^p\,\dx\nonumber\\
	&\quad+
	\iint_{Q_{R,S}}\z^p|D(u-k)_\pm|^p\,\dx\dt\nonumber\\
	&\le
	\gm\iint_{Q_{R,S}}
		\Big[
		(u-k)^{p}_\pm|D\z|^p + (u-k)_\pm^2|\partial_t\z^p|
		\Big]
		\,\dx\dt\\
	&\quad+\iint_{Q_{R,S}}\left(\int_0^{(u-k)_\pm}H'_\varep(k\pm s)s\,\d s\right)|\partial_t\z^p|\,\dx\dt\nonumber\\
	&\quad
	+\frac12\int_{K_R(x_o)\times \{t_o-S\}} \z^p (u-k)_\pm^2\,\dx\nonumber\\
	&\quad+\int_{K_R(x_o)\times\{t_o-S\}}	
	\left(\int_0^{(u-k)_\pm}H'_\varep(k\pm s) s\,\ds\right)\z^p\,\dx\nonumber.
\end{align}
\end{proposition}
The general formulation of \eqref{Eq:en-est-gen} can be simplified, if we take into account the specific structure of $H_\varep$. In particular, since $H'_\varep\ge0$, the second term on the left-hand side can be dropped.
On the other hand, 
since $H_\varep$ is a linear combination of Heaviside functions (an increasing step function) modulo $\varep$,
we have
\[
\int_0^{(u-k)_\pm}H'_\varep(k\pm s) s\,\ds\le (u-k)_\pm \int_0^{(u-k)_\pm}H'_\varep(k\pm s) \,\ds\le \Big(\sum_{i=0}^{\ell}\nu_i\Big)(u-k)_\pm,
\]
provided $\sum_{i=0}^{\ell}\nu_i$ is finite. Instead, if it is infinite, 
we let
\begin{equation*}
M:=\|u\|_{\infty,E_T},
\end{equation*} 
and estimate
\[
\int_0^{(u-k)_\pm}H'_\varep(k\pm s) s\,\ds\le\sup_{-M\le s\le M}|H_\varep(s)|(u-k)_\pm.
\]
Hence, in this case the subsequent estimates will depend also on $M$, but are independent of $\varep$.

With all these remarks, \eqref{Eq:en-est-gen} becomes,
\begin{equation*}
\begin{aligned}
	\essup_{t_o-S<t<t_o}&\frac12\int_{K_R(x_o)\times\{t\}}	
	\z^p (u-k)_\pm^2\,\dx +
	\iint_{Q_{R,S}}\z^p|D(u-k)_\pm|^p\,\dx\dt\\
	&\le
	\gm\iint_{Q_{R,S}}
		\Big[
		(u-k)^{p}_\pm|D\z|^p + (u-k)_\pm^2|\partial_t\z^p|
		\Big]
		\,\dx\dt\\
			&\quad+\gm\iint_{Q_{R,S}}(u-k)_\pm |\pl_t\z^p|\, \dx\dt\\
			&\quad
	+\frac12\int_{K_R(x_o)\times \{t_o-S\}} \z^p (u-k)_\pm^2\,\dx\\
	&\quad  +\gm \int_{K_R(x_o)\times\{t_o-S\}} \z^p (u-k)_\pm \dx,
\end{aligned}
\end{equation*}
where the constant $\gm$ depends only on the data but $M$, if $\sum_{i=0}^{\ell}\nu_i$ is finite. If it is infinite,
the constant $\gm$ also depends on $M$. 


If we choose the cutoff function $\z$ such that $\z(\cdot,t_o-S)=0$, then we obtain
\begin{equation}\label{Eq:energy-bis}
\begin{aligned}
	\essup_{t_o-S<t<t_o}&\frac12\int_{K_R(x_o)\times\{t\}}	
	\z^p (u-k)_\pm^2\,\dx +
	\iint_{Q_{R,S}}\z^p|D(u-k)_\pm|^p\,\dx\dt\\
	&\le
	\gm\iint_{Q_{R,S}}
		\Big[
		(u-k)^{p}_\pm|D\z|^p + (u-k)_\pm^2|\partial_t\z^p|
		\Big]
		\,\dx\dt\\
			&\quad+\gm\iint_{Q_{R,S}}(u-k)_\pm |\pl_t\z^p|\, \dx\dt,
\end{aligned}
\end{equation}
which corresponds to estimate (2.5) of \cite{DBV-95}. 

On the other hand, if we choose the cutoff function such that $\z=\z(x)$, i.e. independent of $t$, we get 
\begin{equation}\label{Eq:energy-ter}
\begin{aligned}
	\essup_{t_o-S<t<t_o}&\frac12\int_{K_R(x_o)\times\{t\}}	
	\z^p (u-k)_\pm^2\,\dx +
	\iint_{Q_{R,S}}\z^p|D(u-k)_\pm|^p\,\dx\dt\\
	&\le
	\gm\iint_{Q_{R,S}}
		(u-k)^{p}_\pm|D\z|^p 	\,\dx\dt\\
			&\quad
	+\frac12\int_{K_R(x_o)\times \{t_o-S\}} \z^p (u-k)_\pm^2\,\dx +\gm \int_{K_R(x_o)\times\{t_o-S\}} \z^p (u-k)_\pm \dx,
\end{aligned}
\end{equation}
which corresponds to estimate (2.6) of \cite{DBV-95}.
\subsection{Logarithmic estimates}
The following logarithmic energy estimate will be useful;
the case $p=2$ has been derived in \cite[(3.13)]{GV-03} (see also \cite[(2.7)]{DBV-95}).
To this end, letting $k$, $u$ and $Q_{R,S}$ be as in Proposition~\ref{Prop:2:1}, we set 
\[
\mathcal{L}:=\sup_{Q_{R,S}}(u-k)_{\pm},
\]
take $c\in(0,\mathcal{L})$, and
 introduce the following function  in $Q_{R,S}$:
\[
\Psi(x,t)\equiv\Psi\big(\mathcal{L},(u-k)_{\pm}, c\big):=\ln_+\Big(\frac{\mathcal{L}}{\mathcal{L}-(u-k)_{\pm}+c}\Big).
\]
This function enjoys the following estimate.
\begin{proposition}\label{Lm:log-est}
Let the hypotheses in Proposition~\ref{Prop:2:1} hold.
There exists $\gm>1$ depending only on the data and on $M$, such that for any $\sig\in(0,1)$,
\begin{align*}
\sup_{t_o-S\le t\le t_o}&\int_{K_{\sig R}(x_o)}\Psi^2(x,t)\,\dx\le \int_{K_{R}(x_o)}\Psi^2(x,t_o-S)\,\dx\\
&+\frac\gm{c}\int_{K_{R}(x_o)}\Psi(x,t_o-S)\,\dx+\frac{\gm}{(1-\sig)^p R^p}\iint_{Q_{R,S}}\Psi |\Psi_u|^{2-p}\,\dx\dt.
\end{align*}
\end{proposition}
\begin{proof}
To simplify the symbolism, we denote $\Psi(s)\equiv \Psi\big(\mathcal{L},s, c\big)$ and its derivative  $\Psi'$.
In the weak formulation we use the test function $\pm\Psi\Psi'\z^p$, with $\z=\z(x)$ and  
\[
\left\{
\begin{aligned}
&\zeta\equiv1\ \text{ on }\ K_{\sig R}(x_o),\\ 
&\zeta=0\ \text{ for }\ x\in\partial K_R(x_o),\\
&|D\zeta|\le\frac1{(1-\sig)R}.
\end{aligned}
\right.
\]
We work in the cylinder $K_R(x_o)\times(t_o-S,t]$ with $t\in(t_o-S,t_o]$. 
Observe that
\[
\pm H_\varep'(u) \partial_t u \Psi\Psi'=\partial_t\int_0^{(u-k)_\pm} H_\varep'(k\pm s)\Psi(s) \Psi'(s)\, \d s.
\]
By the arbitrariness of $t\in(t_o-S,t_o]$, we easily obtain
\begin{align*}
\sup_{t_o-S\le t\le t_o}&\frac12\int_{K_R(x_o)}\Psi^2(x,t)\z^p(x)\,\dx\nonumber\\
&+\sup_{t_o-S\le t\le t_o}\int_{K_R(x_o)\times\{t\}}\left(\int_0^{(u-k)_\pm} H_\varep'(k\pm s)\Psi(s) \Psi'(s)\,\ds\right)\z^p(x)\,\dx\nonumber\\
&+\iint_{Q_{R,S}}(1+\Psi)\Psi'^2|D(u-k)_\pm|^p\z^p\,\dx\d\tau\\
\le&\frac12\int_{K_R}\Psi^2(x,t_o-S)\z^p(x)\,\dx\nonumber\\
&+\int_{K_R(x_o)\times\{t_o-S\}}\left(\int_0^{(u-k)_\pm} H_\varep'(k\pm s)\Psi(s) \Psi'(s)\,\ds\right) \z^p(x)\,\dx\\
&+p\iint_{Q_{R,S}}\Psi\Psi'|D(u-k)_\pm|^{p-1}\z^{p-1}|D\z|\,\dx\d\tau.\nonumber
\end{align*}
Since $H'_\varep\ge0$, the second term on the left-hand side can be discarded.
As for the right-hand side, since $\Psi$ is an increasing function of its argument $(u-k)_\pm$, we have
\begin{align*}
\int_0^{(u-k)_\pm} H_\varep'(k\pm s)\Psi(s) \Psi'(s) \,\ds&\le\Big(\sup_{Q_{R,S}}\Psi'\Big)
\int_0^{(u-k)_\pm}H'_\varep(k\pm s)\Psi(s)\,\ds\\
&\le\Big(\sup_{Q_{R,S}}\Psi'\Big)\Big(\sum_{i=0}^{\ell}\nu_i\Big)\Psi\big((u-k)_\pm\big),
\end{align*}
provided $\sum_{i=0}^{\ell}\nu_i$ is finite. If instead it is infinite, 
we 
estimate
\[
\int_0^{(u-k)_\pm}H'_\varep(k\pm s) \Psi(s) \Psi'(s)\,\ds\le\Big(\sup_{Q_{R,S}}\Psi'\Big)\Big(\sup_{-M\le s\le M}|H_\varep(s)|\Big)\Psi\big((u-k)_\pm\big).
\]
Hence, in this case the subsequent estimates will depend also on $M$.

By its definition, $\Psi'\le 1/c$ in $Q_{R,S}$ and therefore,
\[
\int_{K_R(x_o)\times\{t_o-S\}}\left(\int_0^{(u-k)_\pm} H_\varep'(k\pm s)\Psi \Psi' \,\ds\right)\z^p(x)\,\dx\le\frac\gm{c}\int_{K_R(x_o)}\Psi(x,t_o-S)\,\dx,
\]
where $\gm$ depends only on the data if $\sum_{i=0}^{\ell}\nu_i$ is finite, otherwise it depends also on $M$.
Moreover, since $\Psi\geq0$,
an application of Young's inequality yields that
\begin{align*}
p&\iint_{Q_{R,S}} 
\Psi\Psi' \zeta^{p-1}  |D\umenpm|^{p-1} |D\zeta|\,\dx\d\tau\\
&\le\iint_{Q_{R,S}}(1+\Psi)\Psi'^2\zeta^p |D\umenpm|^p\,\dx\d\tau+\frac\gm{(1-\sigma)^p R^p}\iint_{Q_{R,S}}\Psi|\Psi'|^{2-p}
\,\dx\d\tau.
\end{align*}
Collecting all the terms, we conclude the proof.
\end{proof}
\subsection{De Giorgi type lemmas}\label{S:DG}
For  a cylinder $\mathcal{Q}=K \times(T_1,T_2)\subset E_T$,
we introduce the numbers $\mu^{\pm}$ and $\om$ satisfying
\begin{equation*}
	\mu^+\ge\essup_{\mathcal{Q}} u,
	\quad 
	\mu^-\le\essinf_{\mathcal{Q}} u,
	\quad
	\om\ge\mu^+ - \mu^-.
\end{equation*}

We now present the first De Giorgi type lemma that can be shown by using the energy estimates
in \eqref{Eq:energy-bis}; for the detailed proof we refer to \cite[Lemma~2.1]{Liao-Stefan}.
Here we denote the backward cylinder $(x_o,t_o)+Q_{\rho}(\theta):=K_{\rho}(x_o)\times(t_o-\theta\rho^p,t_o)$.
If no confusion arises, we omit the vertex $(x_o,t_o)$ for simplicity.
\begin{lemma}\label{Lm:DG:1}
Let $u$ be a local weak sub(super)-solution to \eqref{Eq:reg} with \eqref{Eq:1:2} in $E_T$. For $\xi\in(0,1)$,
set $\theta=(\xi\om)^{2-p}$.
There exists a  constant $c_o\in(0,1)$ depending only on the data, such that if
\[
\big|\big[\pm(\mu^\pm -u)\le\xi\om\big]\cap Q_{\rho}(\theta) \big|\le c_o(\xi\om)^{\frac{N+p}p}|Q_{\rho}(\theta)|,
\]
then
\[
\pm(\mu^\pm -u)\ge\tfrac12\xi\om\quad\text{ a.e. in }Q_{\frac12\rho}(\theta),
\]
provided the cylinder $Q_{\rho}(\theta)$ is included in $\mathcal{Q}$. 
\end{lemma}

The next lemma is a variant of the previous one, involving quantitative initial data.
For this purpose, we will use the forward cylinder at $(x_o,t_1)$:
\begin{equation}\label{Eq:forward-cylinder}
(x_o, t_1)+Q^+_{\rho}(\theta):=K_{\rho}(x_o)\times(t_1,t_1+\theta\rho^p).
\end{equation}
We have the following.
\begin{lemma}\label{Lm:DG:initial:1}
Let $u$ be a  local weak sub(super)-solution to \eqref{Eq:reg} with \eqref{Eq:1:2} in $E_T$. 
Assume that for some $\xi\in(0,1)$ there holds
\[
\pm\big(\mu^\pm -u(\cdot, t_1)\big)\ge \xi\om\quad\text{ a.e. in } K_{\rho}(x_o).
\]
There exists a constant $\gm_o\in(0,1)$ depending only on the data, such that for any $\theta>0$, if 
\[
\big|\big[\pm(\mu^\pm -u)\le\xi\om\big]\cap \big[(x_o, t_1)+Q^+_{\rho}(\theta)\big]\big|\le\frac{ \gm_o(\xi\om)^{2-p}}{\theta}|Q^+_{\rho}(\theta)|,
\]
then
\[
\pm(\mu^\pm -u )\ge\tfrac12\xi\om\quad\text{ a.e. in }K_{\frac12\rho}(x_o)\times(t_1,t_1+\theta\rho^p),
\]
provided the cylinder $(x_o, t_1)+Q^+_{\rho}(\theta)$ is included in $\mathcal{Q}$. 
\end{lemma}
\begin{proof}
Let us deal with the case of super-solutions only, as the other case is similar.
We use the energy estimate in \eqref{Eq:energy-ter} in the cylinder $Q_{R,S}\equiv (x_o, t_1)+Q^+_{\rho}(\theta)$, with the 
 levels
\[
k_n=\mu^-+\frac{\xi\om}2+\frac{\xi\om}{2^{n+1}},\quad n=0,1,\cdots.
\] 
Due to this choice of $k_n$ and the assumed pointwise information at $t_1$, the two space integrals at $t_o-S\equiv t_1$ vanish
and the energy estimates reduce to the ones for the parabolic $p$-Laplacian.
As a result, the rest of the proof can be reproduced as in \cite[Chapter~3, \S~4]{DBGV-mono}.
\end{proof}

The next lemma quantifies measure conditions to ensure the degeneracy of the $p$-Laplacian
prevails over the singularity of $\be(\cdot)$. Its proof can be attributed to the theory of parabolic $p$-Laplacian.
Again we omit the vertex of $(x_o,t_o)$ from $Q_{\rho}(\theta)$ for simplicity.
\begin{lemma}\label{Lm:DG:2}
Let $u$ be a  local weak super-solution to \eqref{Eq:reg} with \eqref{Eq:1:2} in $E_T$. 
Assume that for some $\al,\,\eta\in(0,1)$ and $A>1$, there holds
\begin{equation*}
 \big|\big[u(\cdot, t)-\mu^-\ge \eta\om\big]\cap K_\rho\big|>\al |K_\rho|\quad\text{ for all }t\in\big(t_o-A\om^{2-p}\rho^p,t_o\big).
\end{equation*}
 There exists $\xi\in(0,\eta)$, 
such that if $A\ge\xi^{2-p}$ and  
\[
\iint_{Q_{\rho}(\theta)}\int_u^{k} H_\varep'(s)\chi_{[s<k]} \,\d s\dx\dt\le
\xi\om\big|\big[u\le \mu^-+\tfrac12\xi\om\big]\cap Q_{\frac12\rho}(\theta) \big|,
\]
where $k=\mu^-+\xi\om$ and $\theta=(\xi\om)^{2-p}$,
then
\[
u\ge\mu^-+\tfrac14\xi\om\quad\text{ a.e. in } Q_{\frac12\rho}(\theta),
\]
provided the cylinder $Q_{\rho}(A\om^{2-p})$ is included in $\mathcal{Q}$. 
Moreover, it can be traced that $$\xi=\gm(\text{\rm data})\,\eta\exp\big\{-\al^{-\frac{p}{p-1}}\big\}.$$
\end{lemma}
\begin{proof} 
Let us turn our attention to the energy estimate \eqref{Eq:en-est-gen} written with $Q_{R,S}=Q_r(\theta)$
for $\frac12\rho\le r\le \rho$, and with $k=\mu^-+\xi\om$.
 The last integral on the right-hand side is estimated by using the given measure theoretical information:
\begin{align*}
\iint_{Q_{r}(\theta)}&\int_u^{k} H_\varep'(s)(k-s)_+ \,\d s |\pl_t\z^p|\,\dx\dt\\
&\le(k-\mu^-)\iint_{Q_{r}(\theta)}\int_u^{k} H_\varep'(s)\chi_{[s<k]} \,\d s |\pl_t\z^p|\,\dx\dt\\
&\le\xi\om  \|\pl_t\z^p\|_{\infty} \iint_{Q_{\rho}(\theta)}\int_u^{k} H_\varep'(s)\chi_{[s<k]} \,\d s \dx\dt\\
&\le (\xi\om)^2  \|\pl_t\z^p\|_{\infty} \big|\big[u\le \mu^-+\tfrac12\xi\om\big]\cap Q_{\frac12 \rho}(\theta)\big|\\
&\le 4 \|\pl_t\z^p\|_{\infty} \iint_{Q_{r}(\theta)}\big[u-(\mu^-+\xi\om)\big]^2_-\,\dx\dt.
\end{align*}
As such it can be combined with an analogous term involving $\pl_t\z^p$ on the right-hand side
of   \eqref{Eq:en-est-gen}.
Consequently, we end up with an energy estimate of $(u-k)_-$, 
departing from which the theory of parabolic $p$-Laplacian in \cite{DB} applies.
Therefore, we may determine a constant $\xi$ by the data and $\al$, such that
\[
u\ge\mu^-+\tfrac14\xi\om\quad\text{ a.e. in }Q_{\frac12\rho}(\theta).
\]
The dependence of $\xi$ can be traced as in \cite[Appendix~A]{Liao-Stefan}. 
\end{proof}
\begin{remark}\upshape
An analogous statement for sub-solutions holds near $\mu^+$.
Since we do not use it in the argument, it is omitted.
\end{remark}
\subsection{Consequence of the logarithmic estimate}
The setting is the same as in \S~\ref{S:DG}, i.e., we introduce the cylinder $\mathcal{Q}\subset E_T$
and define the quantities $\mu^\pm$ and $\om$ connecting the supremum/infimum and the oscillation of $u$ over $ \mathcal{Q}$.
We will use also cylinders of the forward type \eqref{Eq:forward-cylinder}, with vertex at $(x_o,t_1)$.

The following lemma indicates how the measure of sets where $u$ is close to
the supremum/infimum shrinks at each  level of an arbitrarily long time interval,
 once initial pointwise information is given.
\begin{lemma}\label{Lm:log:1}
Let $u$ be a  local weak sub(super)-solution to \eqref{Eq:1:1} with \eqref{Eq:1:2} in $E_T$. For $\xi\in(0,1)$,
set $\theta=(\xi\om)^{2-p}$. Suppose that 
\begin{equation}\label{IC-log-est}
\pm\big(\mu^{\pm}-u(\cdot, t_1)\big)\ge\xi\om\quad\text{ a.e. in } K_{\rho}(x_o).
\end{equation}
Then for any $\al\in(0,1)$ and $A\ge1$, there exists $\bar\xi\in(0,\frac14\xi)$, 
such that
\[
\big|\big[\pm\big(\mu^{\pm}-u(\cdot, t)\big)\le \bar{\xi}\om\big]\cap K_{\frac12\rho}(x_o)\big|\le \al |K_{\frac12\rho}|\quad\text{ for all }t\in(t_1,t_1+A\theta\rho^p),
\]
provided the cylinder $K_{\rho}(x_o)\times(t_1,t_1+A\theta\rho^p)$ is included in $\mathcal{Q}$.
Moreover, the dependence of $\bar\xi$ is given by
\[
\bar\xi=\tfrac12\xi\exp\Big\{-\gm(\text{\rm data})\frac A\al\Big\}.
\]
\end{lemma}
\begin{proof}
We will prove the estimate with $\mu^+$, since the one with $\mu^-$ is completely analogous. Moreover, for simplicity we omit the dependence on $x_o$.  Proposition~\ref{Lm:log-est} will be used in the cylinder $K_{\rho}(x_o)\times(t_1,t_1+A\theta\rho^p) $.
To this end, let us put
$$k=\mu^+-\xi\om,\qquad\sigma=\tfrac12,\qquad c=\bar\xi\om,$$
with $\bar\xi\in(0,\frac14\xi)$ to be chosen. Due to \eqref{IC-log-est} the integrals on the right-hand side at time $t=t_1$ vanish. Therefore, we are left with 
$$\sup_{t_1\leq t\leq{t_1+A\theta\rho^p}}\int_{K_{\frac12\rho}}\Psi^2(x,t)\,\dx\leq\frac{4\gamma}{\rho^p}
\int_{t_1}^{t_1+A\theta\rho^p}\int_{K_\rho}\Psi|\Psi_u|^{2-p}\,\dx\d\tau.$$
Let us relabel $4\gamma$ as $\gamma$. It is easy to see that 
$$\Psi\leq\ln \frac{\mathcal L}{\bar\xi\om} \le\ln\frac{\xi}{\bar\xi},$$
since $${\mathcal L}=\sup_{\mathcal{Q}}\umenkp\leq\mu^+-\mu^++\xi\om=\xi\om.$$ On the other hand,
\[
|\Psi_u|^{2-p}=\left|{\mathcal L}-(u-k)_++\bar\xi\om\right|^{p-2}\le(2\xi\om)^{p-2}.
\]
Hence, we may estimate
\begin{align*}
\frac\gm{\rho^p}\int_{t_1}^{t_1+A\theta\rho^p}\int_{K_\rho}\Psi|\Psi_u|^{2-p}\,\dx\d\tau
&\le\gm A\theta(2\xi\om)^{p-2}\,|K_\rho| \Big(\ln \frac\xi{\bar\xi}\Big) \le \gm A\,|K_\rho|
\Big(\ln\frac\xi{2\bar\xi} \Big),
\end{align*}
bearing in mind that $\bar\xi\in(0,\frac14\xi)$ and $\theta=(\xi\om)^{2-p}$. If we consider
$$A_{\bar\xi,{\frac12\rho}}(t):=\big[u(\cdot,t)>\mu^+-\bar\xi\om\big] \cap K_{\frac12\rho}$$
as integration set for the integral on the left-hand side, instead of the larger $K_{{\frac12\rho}}$
and note that $\Psi$ is decreasing in $\mathcal{L}$, we may estimate over $A_{\bar\xi,{\frac12\rho}}(t)$:
\[
\Psi\ge\ln\frac{\xi\om}{\xi\om-(\xi\om-\bar\xi\om)+\bar\xi\om}=\ln\frac{\xi}{2\bar\xi}.
\]
%
Then we obtain
$$\Big(\ln\frac\xi{2\bar\xi}\Big)^2|A_{\bar\xi,{\frac12\rho}}(t)|
\leq\widetilde\gamma A \Big(\ln\frac\xi{2\bar\xi}\Big) |K_{\frac12\rho}|\quad\text{ for all } t\in(t_1,{t_1+A\theta\rho^p})$$
with $\widetilde\gamma=2\gm$, that is
$$|A_{\bar\xi,{\frac12\rho}}(t)|\le \frac{ \widetilde\gamma A}{\ln\left(\xi/{2\bar\xi}\right)}
|K_{\frac12\rho}|\quad\text{ for all } t\in(t_1,{t_1+A\theta\rho^p}).$$
If we choose $\bar\xi$ such that
\[
\al\equiv\frac{\widetilde\gamma A}{\ln\left(\xi/2\bar\xi\right)},
\]
we conclude the proof.
%
\end{proof}
%

\section{Proof of Theorem~\ref{Thm:1:1}}
Assume $(x_o,t_o)=(0,0)$, introduce $Q_o=K_\rho\times(-\rho^{p-1},0)$  and
set
\begin{equation*}
	\mu^+=\essup_{Q_o}u,
	\quad
	\mu^-=\essinf_{Q_o}u,
	\quad
	\om \geq \mu^+ - \mu^-.
\end{equation*}
Letting $\theta=(\frac14 \om)^{2-p}$, for some $A(\om)>1$ to be determined in terms of the data and $\om$,
we may assume that
\begin{equation}\label{Eq:start-cylinder}
Q_\rho(A\theta)\subset Q_o=K_\rho\times(-\rho^{p-1},0),\quad\text{ such that }\quad \essosc_{Q_{\rho}(A\theta)}u\le \om;
\end{equation}
the case when the set inclusion does not hold will be incorporated later.

%
%

\subsection{Reduction of oscillation near the supremum}\label{S:1st-alt}
In this section, we work with $u$ as a sub-solution near its supremum. 
Recalling that $\theta=(\frac14 \om)^{2-p}$, suppose for some $\bar{t}\in\big(-(A-1)\theta\rho^p,0\big)$, there holds
\begin{equation}\label{Eq:1st-alt}
	\big|\big[\mu^+ - u\le \tfrac14  \om\big]
	\cap 
	\big[(0,\bar{t})+Q_{\rho}(\theta)\big]\big|\le c_o(\tfrac14 \om)^{\frac{N+p}{p}}|Q_{\rho}(\theta)|=:\al|Q_{\rho}(\theta)|,
\end{equation}
where $c_o$ is the constant determined in Lemma~\ref{Lm:DG:1} in terms of the data.
According to Lemma~\ref{Lm:DG:1} (with $\xi=\frac14$),
we have 
\begin{equation}\label{Eq:pt-est-1}
	\mu^+ - u\ge \tfrac{1}8 \om
	\quad
	\mbox{a.e.~in $(0,\bar{t})+Q_{\frac12 \rho}(\theta)$.}
\end{equation}
This pointwise estimate \eqref{Eq:pt-est-1} at $t_1:=\bar{t}-\theta(\tfrac12\rho)^p$ serves as the initial datum 
for an application of Lemma~\ref{Lm:DG:initial:1} and Lemma~\ref{Lm:log:1}.
First of all, according to Lemma~\ref{Lm:DG:initial:1}, there exists $\gm_o\in(0,1)$, such that
if for some $\eta\in(0,\frac18)$,
\begin{equation}\label{Eq:shrink-1}
\big|\big[\mu^+ - u\le \eta \om \big]\cap \widetilde{Q}  \big|
\le  \frac{\gm_o(\tfrac18\om)^{2-p}}{A\theta}|\widetilde{Q} |
\quad\text{ where }\widetilde{Q} :=K_{\frac12\rho}\times (t_1,0),
\end{equation}
then 
\begin{equation}\label{Eq:pt-est-2}
\mu^+ - u\ge \tfrac{1}{2}\eta \om\quad\text{ a.e. in }K_{\frac14\rho}\times (t_1,0).
\end{equation}
On the other hand, 
owing to \eqref{Eq:pt-est-1},  Lemma~\ref{Lm:log:1} implies that \eqref{Eq:shrink-1} is verified with the choice
  \[
  \eta=\tfrac1{16}\exp\Big\{-\frac{\gm A^2}{2^{p-2}\gm_o}\Big\},
  \]
and hence so is \eqref{Eq:pt-est-2} due to Lemma~\ref{Lm:DG:initial:1}. Note that all constants are stable as $p\downarrow2$.
Consequently, the estimate \eqref{Eq:pt-est-2} yields a reduction of oscillation
\begin{equation}\label{Eq:reduc-osc-0}
\essosc_{Q_{\frac14\rho}(\theta)}u\le\big(1-\tfrac12\eta\big)\om.
\end{equation}
Keep in mind that $A(\om)$ is yet to be chosen.

\subsection{Reduction of oscillation near the infimum: Part~I}\label{S:2nd-alt}
Starting from this section, let us suppose \eqref{Eq:1st-alt} does not hold, that is, for any 
$\bar{t}\in\big(-(A-1)\theta\rho^p,0\big]$,
\begin{equation*}
	\big|\big[\mu^+ - u\le \tfrac14  \om\big]
	\cap 
	\big[(0,\bar{t})+Q_{\rho}(\theta)\big]\big|>\al|Q_{\rho}(\theta)|\quad\text{ for }\al=c_o(\tfrac14 \om)^{\frac{N+p}{p}}.
\end{equation*}
Since $\mu^+-\frac14\om\ge\mu^-+\frac14\om$  may be assumed without loss of generality, we rephrase it as
\begin{equation}\label{Eq:2nd-alt}
\big|\big[u-\mu^-\ge\tfrac14\om\big]\cap \big[(0,\bar{t})+Q_{\rho}(\theta)\big] \big|> \al|Q_{\rho}(\theta)|.
\end{equation}

Fixing such $\bar{t}$ for the moment, we will analyze the local clustering phenomenon of $u$
encoded in the measure information \eqref{Eq:2nd-alt}. The idea of the following argument is
taken from  \cite[\S\S~5 -- 8]{DBV-95}. We will work with $u$ as a super-solution near its infimum
throughout \S\S~\ref{S:2nd-alt} -- \ref{S:reduc-osc-4}.

\begin{lemma}\label{Lm:cluster}
For every $\lm\in(0,1)$ and $\eta\in(0,1)$, there exist a point $(x_*,t_*)\in (0,\bar{t})+Q_{\rho}(\theta)$,
a number $\kappa\in(0,1)$ and a cylinder $(x_*,t_*)+Q_{\kappa\rho}(\theta)\subset (0,\bar{t})+Q_{\rho}(\theta)$, such that
\[
\big|\big[u\le\mu^-+\tfrac14\lm\om\big]\cap \big[(x_*,t_*)+Q_{\kappa\rho}(\theta)\big]\big| \le \eta |Q_{\kappa\rho}(\theta)|.
\]
The dependence of $\kappa$ is traced by $\kappa=\gm(\text{\rm data}) (1-\lm)^{3+\frac2p}\eta^{1+\frac1p}\al^{2+\frac1p}\om^{1+\frac1p}$.
\end{lemma}
\begin{proof}
For simplicity of notation, we take $\bar{t}=0$. 
 Let $\z$ be a standard cutoff function in $Q_{2\rho}(\theta)$ that vanishes on its parabolic boundary
and equals the identity in $Q_{\rho}(\theta)$, satisfying $|D\z|\le \gm/\rho$ and $|\pl_t \z|\le \gm/(\theta\rho^p) $.
According to the energy estimate \eqref{Eq:energy-bis} written in $Q_{2\rho}(\theta)$ for $(u-k)_-$
with $k=\mu^-+\frac14\om$ and with the cutoff function $\z$, a simple calculation yields
\[
\iint_{Q_{\rho}(\theta)}|D(u-k)_-|^p\,\dx\dt\le\frac{\gm}{\rho^p}\om^p\Big(1+\frac1{\om}\Big) |Q_{\rho}(\theta)|.
\]
In terms of $$v:=\frac{\tfrac14\om-(u-k)_-}{\tfrac14\om},$$
 this energy estimate may be written as
\begin{equation}\label{Eq:energy-v}
\iint_{Q_{\rho}(\theta)}|Dv|^p\,\dx\dt\le\frac{\gm}{\rho^p}\Big(1+\frac1{\om}\Big) |Q_{\rho}(\theta)|,
\end{equation}
and meanwhile the measure information \eqref{Eq:2nd-alt} reads
\begin{equation}\label{Eq:2nd-alt-v}
\big|\big[v\ge1\big]\cap Q_{\rho}(\theta)\big|> \al|Q_{\rho}(\theta)|.
\end{equation}

To proceed, let us define the set
\[
A(t):=\big[v(\cdot, t)\ge 1\big]\cap K_\rho,
\]
and the set
\[
\mathcal{P}:=\big\{t\in(-\theta\rho^p,0): |A(t)|\ge\tfrac12\al |K_\rho|\big\}.
\]
Now we may estimate by using \eqref{Eq:2nd-alt-v}:
\begin{align*}
\al |Q_{\rho}(\theta)|\le\int_{-\theta\rho^p}^{0}|A(t)|\,\dt &=\int_{\mathcal{P}}|A(t)|\,\dt + \int_{\mathcal{P}^c}|A(t)|\,\dt
\le \big(|\mathcal{P}|+\tfrac12\al\theta\rho^p\big)|K_\rho|,
\end{align*}
which implies $|\mathcal{P}|\ge\tfrac12\al\theta\rho^p$. This joint with \eqref{Eq:energy-v} yields that
\[
\tfrac12\al\theta\rho^p\inf_{t\in\mathcal{P}}\int_{K_\rho \times\{t\}}|Dv|^p\,\dx\le \iint_{Q_\rho(\theta)}|Dv|^p\,\dx\dt\le
\frac{\gm}{\rho^p} \Big(1+\frac1{\om}\Big) |Q_{\rho}(\theta)|,
\]
that is,
\[
\inf_{t\in\mathcal{P}}\int_{K_\rho \times\{t\}}|Dv|^p\,\dx\le \frac{2\gm}{\al\rho^p}\Big(1+\frac1{\om}\Big)|K_\rho|.
\]
Therefore, there exists $\widetilde{t}\in\mathcal{P}$, such that
\begin{equation}\label{Eq:cluster-condi-1}
\int_{K_\rho\times\{\widetilde{t}\}}|Dv|^p\,\dx\le\frac{2\gm}{\al\rho^p}\Big(1+\frac1{\om}\Big)|K_\rho|.
\end{equation}
Meanwhile, by the definition of $\mathcal{P}$, there holds that
\begin{equation}\label{Eq:cluster-condi-2}
\big|\big[v(\cdot, \widetilde{t})\ge 1\big]\cap K_\rho\big|\ge\tfrac12\al |K_\rho|.
\end{equation}
Based on \eqref{Eq:cluster-condi-1} and \eqref{Eq:cluster-condi-2}, 
we are ready to apply \cite[Chapter~2, Lemma~3.1]{DBGV-mono}.
Indeed, let $\widetilde\lm:=\tfrac12(1+\lm)$ and $\widetilde\eta\in(0,1)$ to be determined:
there exist $\widetilde{x}\in K_\rho$
and 
\begin{equation}\label{Eq:choice-eps}
\varep=\gm(\text{data})(1-\widetilde\lm)\widetilde\eta\al^{2+\frac1p}\om^{\frac1p},
\end{equation} 
such that
\[
\big|\big[v(\cdot, \widetilde{t})\ge\widetilde{\lm}\big]\cap K_{\varep\rho}(\widetilde{x})\big|\ge(1-\widetilde{\eta}) |K_{\varep\rho}|.
\]
Reverting to $u$, we actually obtain
\begin{equation}\label{Eq:measure-u}
\big|\big[u(\cdot, \widetilde{t})\le\mu^-+\tfrac14\widetilde{\lm}\om\big]\cap K_{\varep\rho}(\widetilde{x})\big|
\le\widetilde{\eta} |K_{\varep\rho}|.
\end{equation}

In order to propagate this measure information, we consider the forward cylinders
\[
\left\{
\begin{aligned}
& K_{\frac12\varep\rho}(\widetilde{x})\times\big(\widetilde{t},\widetilde{t}+\dl\theta(\varep\rho)^p\big),\\
& K_{ \varep\rho}(\widetilde{x})\times\big(\widetilde{t},\widetilde{t}+\dl\theta(\varep\rho)^p\big),
\end{aligned}
\right.
\]
where $\dl>0$ is to be determined. Let $\z(x)$ be a cutoff function in $K_{ \varep\rho}(\widetilde{x})$
that vanishes on $\pl K_{ \varep\rho}(\widetilde{x})$ and equals the identity in $K_{ \frac12\varep\rho}(\widetilde{x})$,
such that $|D\z|\le\gm/ (\varep\rho)$. Employing \eqref{Eq:measure-u}, the energy estimate \eqref{Eq:energy-ter} for $(u-k)_-$
with $k=\mu^-+\tfrac14\widetilde{\lm}\om$ in this setting gives that
\[
\int_{K_{ \frac12\varep\rho}(\widetilde{x})\times\{t\}}(u-k)_-^2\,\dx\le\gm\om^2\Big(\dl+\frac{2\widetilde\eta}{\om}\Big) 
|K_{ \frac12\varep\rho}|,
\]
for all $t\in\big(\widetilde{t},\widetilde{t}+\dl\theta( \varep\rho)^p\big)$.

We estimate the integral on the left-hand side from below by
\begin{align*}
\int_{K_{ \frac12\varep\rho}(\widetilde{x})\times\{t\}}(u-k)_-^2\,\dx
&\ge\int_{K_{ \frac12\varep\rho}(\widetilde{x})\times\{t\}}(u-k)_-^2\chi_{[u<\mu^-+\frac14\lm\om]}\,\dx\\
&\ge\tfrac1{64}(1-\lm)^2\om^2 \big|\big[u(\cdot, t)<\mu^-+\tfrac14\lm\om\big]\cap K_{ \frac12\varep\rho}(\widetilde{x})\big|.
\end{align*}
Substituting this estimate back to the energy estimate yields that
\[
\big|\big[u(\cdot, t)<\mu^-+\tfrac14\lm\om\big]\cap K_{ \frac12\varep\rho}(\widetilde{x})\big|\le
\frac{\gm}{(1-\lm)^2}\Big(\dl+\frac{2\widetilde\eta}{\om}\Big) |K_{ \frac12\varep\rho}|.
\]
Now we may choose $\dl$ and $\widetilde\eta$ to satisfy
\begin{equation}\label{Eq:choice-dl}
\frac{\gm}{(1-\lm)^2}\Big(\dl+\frac{2\widetilde\eta}{\om}\Big)\le\eta.
\end{equation}
Up to now, we have shown that
\[
\big|\big[u(\cdot, t)\ge\mu^-+\tfrac14\lm\om\big]\cap K_{ \frac12\varep\rho}(\widetilde{x})\big|\ge
(1-\eta) |K_{ \frac12\varep\rho}|,
\]
for all $t\in\big(\widetilde{t},\widetilde{t}+\dl\theta(\varep\rho)^p\big)$.
For simplicity let us denote $r:=\frac12\varep\rho$. The above slicewise measure information actually yields 
\begin{equation}\label{Eq:cluster-1}
\big|\big[u \ge\mu^-+\tfrac14\lm\om\big]\cap Q \big|\ge
(1-\eta) |Q|,\quad\text{ where } Q:=K_{r}(\widetilde{x})\times\big(\widetilde{t},\widetilde{t}+\dl\theta(\varep\rho)^p\big).
\end{equation}
Arranging $L:=\tfrac12\dl^{-\frac1p}$ to be an integer, we partition $Q$ along the space coordinate planes
 into $L^N$ disjoint but adjacent sub-cylinders, each of which is congruent to
 \[
 K_{\frac{r}{L}}\times\big(-\dl\theta(\varep\rho)^p,0\big)\equiv K_{\kappa \rho}\times\big(-\theta(\kappa\rho)^p,0\big)=Q_{\kappa\rho}(\theta)
 \]
 where $\kappa:=\varep\dl^{\frac1p}$ can be traced by combining \eqref{Eq:choice-eps} and  \eqref{Eq:choice-dl}, i.e.
 \[
\kappa=\gm(\text{data}) (1-\lm)^{3+\frac2p} \eta^{1+\frac1p} \al^{2+\frac1p}\om^{1+\frac1p}.
 \]
 Due to \eqref{Eq:cluster-1}, it is easy to see that at least one of them, say $(x_*,t_*)+Q_{\kappa\rho}(\theta)$, will satisfy the desired property
 \[
 \big|\big[u \ge\mu^-+\tfrac14\lm\om\big]\cap \big[(x_*,t_*)+Q_{\kappa\rho}(\theta) \big]\big|\ge(1-\eta)| Q_{\kappa\rho}(\theta)|.
 \]
The proof is concluded with such a choice of $(x_*,t_*)$ and  $\kappa$.
\end{proof}

The location of the  clustering within $(x_*,t_*)+Q_{\kappa\rho}(\theta)\subset(0,\bar{t})+Q_{\rho}(\theta)$ 
being only qualitative notwithstanding, the quantified measure concentration
allows us to extract pointwise estimate with the aid of Lemma~\ref{Lm:DG:1} and then use
the logarithmic energy estimate to propagate the measure information up to the level $\bar{t}$, cf.~Figure~\ref{Fig-2}.

\begin{figure}[t]\label{Fig-2}
\centering
\includegraphics[scale=1]{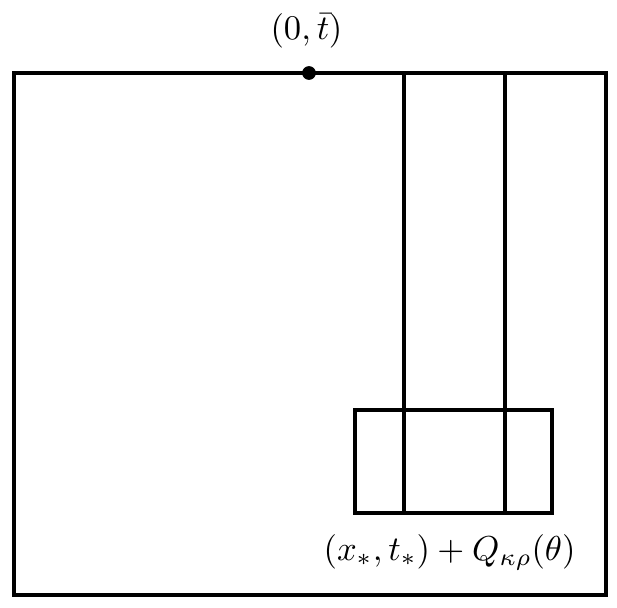}
\caption{Local clustering}
\end{figure}

As a matter of fact, if in Lemma~\ref{Lm:cluster}  we choose $\lm=\tfrac12$ and 
$\eta=c_o(\tfrac14\om)^{\frac{N+p}p}$ where $c_o$ is determined in Lemma~\ref{Lm:DG:1},
then Lemma~\ref{Lm:cluster} and Lemma~\ref{Lm:DG:1} would yield that
\[
u\ge\mu^-+\tfrac1{16}\om\quad\text{ a.e. in }(x_*,t_*)+Q_{\frac12\kappa\rho}(\theta),
\]
for some $(x_*,t_*)\in (0,\bar{t})+Q_{\rho}(\theta)$ and the constant 
\[
\kappa=\gm(\text{data})\om^{\bar{q}}\quad\text{ where }\bar{q}:=\big(\tfrac{N}p+1\big)\big(3+\tfrac2p\big)+\tfrac1p+1.
\]
In particular, we have 
\[
u\big(\cdot, t_*-\theta(\tfrac12\kappa\rho)^p\big)\ge\mu^-+\tfrac1{16}\om \quad\text{ a.e. in }K_{\frac12\kappa\rho}(x_*),
\]
which serves as the initial datum to apply Lemma~\ref{Lm:log:1}.
Indeed, setting $\al=\tfrac12$ and $\xi=\tfrac1{16}$ in Lemma~\ref{Lm:log:1}
and choosing $\widetilde{A}$ so large that
\[
(\tfrac14\om)^{2-p}\rho^p\le \widetilde{A}(\tfrac1{16}\om)^{2-p}(\tfrac12\kappa\rho)^p,\quad\text{ i.e. }\quad \widetilde{A}\ge\frac{2^{4-p}}{\kappa^{p}},
\]
 it yields a number $\bar\xi\in(0,\tfrac1{4}\xi)$, such that
\begin{equation}\label{Eq:measure-bar}
\big|\big[u(\cdot, \bar{t})>\mu^-+\bar\xi\om\big]\cap K_{\frac14\kappa\rho}(x_*) \big|>\tfrac12 |K_{\frac14\kappa\rho}|.
\end{equation}
The dependence of $\bar\xi$ is traced by
\begin{equation}\label{Eq:choice-xi-bar}
\bar\xi=\tfrac12\xi\exp\Big\{- \frac{\gm\widetilde{A}}\al\Big\}=\tfrac1{32}\exp\Big\{-\frac{\gm2^{4-p}}{\kappa^p} \Big\}=\tfrac1{32}\exp\Big\{ -\frac{\gm}{\om^{p\bar{q}}} \Big\}.
\end{equation}
The measure information \eqref{Eq:measure-bar} permits us to claim that
\begin{align*}
\big|\big[u(\cdot, \bar{t})>\mu^-+\bar\xi\om\big]\cap K_{ \rho}  \big|&\ge
\big|\big[u(\cdot, \bar{t})>\mu^-+\bar\xi\om\big]\cap K_{\frac14\kappa\rho}(x_*) \big|\\
&>\tfrac12 |K_{\frac14\kappa\rho}|=\tfrac12(\tfrac14\kappa)^N |K_\rho|=:\bar\al |K_\rho|.
\end{align*}
Thanks to the arbitrariness of $\bar{t}$, we have actually arrived at
\begin{equation}\label{Eq:measure:0}
 \big|\big[u(\cdot, t)\ge\mu^- + \bar\xi\om\big]\cap K_\rho\big|>\bar\al |K_\rho|\quad\text{ for all }t\in\big(-(A-1) \theta\rho^p,0\big].
\end{equation}
The dependence of $\bar\al$ is traced by 
\begin{equation}\label{Eq:choice-beta}
\bar\al=\gm(\text{data})\,\om^{\bar{q}N}.
\end{equation}

This measure information \eqref{Eq:measure:0} lays the foundation for the analysis to be set out in the following sections.
Since $A$ is a large number, we will stipulate that \eqref{Eq:measure:0} holds with $A-1$ replaced by $A$ for simplicity.
\subsection{Reduction of oscillation near the infimum: Part~II}\label{S:reduc-osc-1}
Let us first introduce the following intrinsic cylinders 
\begin{equation*}
\left\{
\begin{array}{ll}
Q_{\rho}(\widehat\theta)=K_\rho \times( -  \widehat\theta\rho^p,0),\quad \widehat\theta=(\xi\om)^{2-p},\\[5pt]
Q_{\rho}(\widetilde\theta)=K_\rho \times( -\widetilde\theta\rho^p,0),\quad \widetilde\theta=(\dl\xi\om)^{1-p},
\end{array}\right.
\end{equation*}
for some $\xi(\om)$ and $\dl(\om)$ in $(0,1)$ to be determined later.
We can always assume $\xi$ and $\dl$ to be sufficiently small, so that $ \widehat\theta\le\widetilde{\theta}$. 
On the other hand, we may assume that 
\begin{equation}\label{Eq:A-control}
8^p\widetilde{\theta}\le A  \theta
\end{equation}
 for some $A(\om)$ yet to be determined.

Throughout \S\S~\ref{S:reduc-osc-1} -- \ref{S:reduc-osc-3}, 
we always assume that
\begin{equation}\label{Eq:mu-minus}
|\mu^- - e_i|\le\tfrac14\dl\xi\om \quad\text{ for some  } i\in\{0,1,\cdots,\ell\},
\end{equation} 
for  the same $\xi(\om)$ and $\dl(\om)$ in $(0,1)$ introduced above, to be determined.
When restriction \eqref{Eq:mu-minus} does not hold, the case will be examined in \S~\ref{S:reduc-osc-4}.

First of all, we turn our attention to
Lemma~\ref{Lm:DG:1} and Lemma~\ref{Lm:DG:2}. 
In view of the measure information \eqref{Eq:measure:0}, Lemma~\ref{Lm:DG:2} is at our disposal, with $\al$, $\eta$ 
and $A$ replaced by $\bar\al$, $\bar\xi$ and $A/4^{2-p}$ respectively. Suppose $\xi$ is determined in Lemma~\ref{Lm:DG:2} in terms of the data and $\bar\al$ fixed in \eqref{Eq:choice-beta},
and recall that $ \widehat\theta=(\xi\om)^{2-p}$.
If there holds
\[
\big|\big[u\le\mu^-+\tfrac12\xi\om\big]\cap Q_{\frac12\rho}( \widehat\theta)\big|\le c_o(\xi\om)^{\frac{N+p}p}|Q_{\frac12\rho}( \widehat\theta)|,
\]
then  Lemma~\ref{Lm:DG:1} yields that
\begin{equation}\label{Eq:lower-bd-1}
u\ge \mu^-+\tfrac14\xi\om\quad\text{ a.e. in }Q_{\frac14\rho}( \widehat\theta).
\end{equation}
Analogously, if for $k=\mu^-+ \xi\om$, there holds
\begin{equation*}
\iint_{Q_{\rho}( \widehat\theta)}\int_u^{k} H_\varep'(s) \chi_{[s<k]}\,\d s\dx\dt\le\xi\om\big|\big[u\le \mu^-+\tfrac12\xi\om\big]\cap Q_{\frac12\rho}( \widehat\theta)\big|
\end{equation*}
then Lemma~\ref{Lm:DG:2} yields that, stipulating $A\ge 4^{2-p}\xi^{2-p}$, 
\begin{equation}\label{Eq:lower-bd-2}
u\ge \mu^-+\tfrac14\xi\om\quad\text{ a.e. in }Q_{\frac12\rho}( \widehat\theta).
\end{equation}
Consequently, either \eqref{Eq:lower-bd-1} or \eqref{Eq:lower-bd-2} yields a reduction of oscillation
\begin{equation}\label{Eq:reduc-osc-1}
\essosc_{Q_{\frac14\rho}( \widehat\theta)}u\le(1-\tfrac14\xi)\om.
\end{equation}
For later use, we record the dependence of $\xi$ here, that is,
\begin{equation}\label{Eq:choice-xi}
\xi= \exp\Big\{- \gm(\text{data})\om^{-\frac{p\bar{q}N}{p-1}}-\gm(\text{data})\om^{-p\bar{q}}\Big\}.
\end{equation}
\color{red}
\color{black}
\subsection{Reduction of oscillation near the infimum: Part~III}\label{S:reduc-osc-2}
In this section, we continue to examine the situation when the measure condition in Lemma~\ref{Lm:DG:1} is violated:
\begin{equation}\label{Eq:opposite:1}
\big|\big[u\le\mu^-+\tfrac12\xi\om\big]\cap Q_{\frac12\rho}( \widehat\theta)\big|> c_o(\xi\om)^{\frac{N+p}p}|Q_{\frac12\rho}( \widehat\theta)|,
\end{equation}
and when the condition in Lemma~\ref{Lm:DG:2} is also violated: for $k=\mu^-+ \xi\om$, 
there holds
\begin{equation}\label{Eq:opposite:2}
 \iint_{Q_{\rho}( \widehat\theta)}\int_u^{k} H_\varep'(s)\chi_{[s<k]} \,\d s\dx\dt>\xi\om\big|\big[u\le \mu^-+\tfrac12\xi\om\big]\cap Q_{\frac12\rho}(\widehat\theta)\big|.
\end{equation}

Combining \eqref{Eq:opposite:1} and \eqref{Eq:opposite:2}, we  obtain that, for all $r\in[2\rho, 8\rho]$,
\begin{equation}\label{Eq:opposite:3}
\begin{aligned}
\iint_{Q_{r}(\widehat\theta)}\int_u^{k} H_\varep'(s)\chi_{[s<k]} \,\d s\dx\dt&\ge\iint_{Q_{\rho}(\widehat\theta)}\int_u^{k} H_\varep'(s)\chi_{[s<k]} \,\d s\dx\dt\\
&\ge\xi\om\big|\big[u\le \mu^-+\tfrac12\xi\om\big]\cap Q_{\frac12\rho}(\widehat\theta)\big|\\
&\ge c_o(\xi\om)^b|Q_{\frac12\rho}(\widehat\theta)|\ge\widetilde{\gm}(\xi\om)^b|Q_r(\widehat\theta)|,
\end{aligned}
\end{equation}
where $\widetilde{\gm}=c_o16^{-N-p}$ and
$b=1+\tfrac{N+p}p$.

Next, introduce a free parameter $\bar\dl\in(\dl,2\dl)$ and set $\bar\theta:=(\bar\dl\xi\om)^{1-p}$.
Recall also that $ \theta=(\tfrac14\om)^{2-p}$, $\widetilde{\theta}=(\dl\xi\om)^{1-p}$, $\widehat\theta=(\xi\om)^{2-p}$,
and  that we have assumed $\widetilde{\theta}(8\rho)^p\le A \theta\rho^p\le\rho^{p-1}$ in \eqref{Eq:A-control}. 
Therefore, 
\[
Q_r( \widehat\theta)\subset Q_r(\bar{\theta})\subset Q_r(\widetilde{\theta})\subset Q_{r}(A \theta)\subset \widetilde{Q}_o
\quad\text{ for any } r\in[2\rho, 8\rho].
\]
The estimate \eqref{Eq:opposite:3} implies that there exists $t_*\in [-\widehat\theta r^p,0]$,
such that
\begin{equation}\label{Eq:time:1}
\int_{K_{r}\times\{t_*\}}\int_u^{k} H_\varep'(s)\chi_{[s<k]} \,\d s\dx\ge\widetilde{\gm}(\xi\om)^b|K_r|.
\end{equation}
 Observe also that for any $t\in[-\bar{\theta} r^p, 0]$ and any  $\bar\dl\in(\dl,2\dl)$, there holds
 \begin{equation}\label{Eq:time:2}
\begin{aligned}
|K_r|&\ge\big|\big[u\le\mu^-+\bar\dl\xi\om\big]\cap K_r\big|\\
&\ge(\bar\dl\xi\om)^{-2}\int_{K_r\times\{t\}}\big[u-(\mu^-+\bar\dl\xi\om)\big]_-^2\,\dx.
\end{aligned}
\end{equation}
Denoting $\bar{k}=\mu^-+ \bar\dl\xi\om$ and  enforcing that for some $i\in\{0,1,\cdots,\ell\}$,
\begin{equation*}
|\mu^- - e_i|\le\tfrac14\dl\xi\om\quad\text{ and }\quad \varep\le\tfrac14\dl\xi\om,
\end{equation*}
we use \eqref{Eq:time:1} and \eqref{Eq:time:2} to estimate
\begin{equation*}
\begin{aligned}
&\essup_{-\bar{\theta}r^p<t<0}\int_{K_{r}}\int_u^{\bar{k}} H_\varep'(s)(s-\bar{k})_-\,\d s \dx\\
&\qquad\ge(\bar{k}-e_i- \varep )\essup_{-\bar{\theta}r^p<t<0}\int_{K_{r}}\int_u^{\bar{k}} H_\varep'(s)\chi_{[s<\bar{k}]}\,\d s \dx\\
&\qquad\ge\widetilde{\gm}(\bar{k}-e_i-\varep)(\xi\om)^{b}|K_{r}|\\
&\qquad\ge \tfrac12\widetilde{\gm}(\dl\xi\om)(\xi\om)^{b}
(\bar\dl\xi\om)^{-2}\essup_{-\bar{\theta} r^p<t<0}\int_{K_r\times\{t\}}\big[u-(\mu^-+\bar\dl\xi\om)\big]_-^2\,\dx.
\end{aligned}
\end{equation*}
In the first inequality above, we have assumed  $\tfrac94\xi\om\le d$ by possibly further restricting
 the choice of $\xi$ in \eqref{Eq:choice-xi}, and hence
$(\mu^-,\bar{k})\subset (e_i-\tfrac14\dl\xi\om,e_i+\tfrac14\dl\xi\om+2\dl\xi\om)\subset (e_i-d,e_i+d)$.
As such the constant $\gm$ in the definition of $\xi$ in \eqref{Eq:choice-xi} depends on $d$
and $M$.
The above analysis together with \eqref{Eq:en-est-gen}
yields the following  energy estimate.
\begin{lemma}\label{Lm:energy}
Let $u$ be a  weak super-solution to \eqref{Eq:reg} with \eqref{Eq:1:2} in $E_T$, under the measure information \eqref{Eq:measure:0}.
Let \eqref{Eq:opposite:1} and \eqref{Eq:opposite:2} hold true.
Denoting $b:=1+\tfrac{N+p}p$ and 
setting $k=\mu^-+\bar\dl\xi\om$ with  $\bar\dl\in(\dl,2\dl)$, there exists a positive constant $\gm$
depending only on the data, such that  for all $\sig\in(0,1)$ and all $r\in[2\rho, 8\rho]$ we have
\begin{align*}
\dl\xi\om( \bar\dl\xi\om)^{-2}(\xi\om)^b&\essup_{ -\bar{\theta}(\sig r)^p<t<0}\int_{K_{\sig r}\times\{t\}}(u-k)_-^2\,\dx
+\iint_{Q_{\sig r}(\bar{\theta})}|D(u-k)_-|^p\,\dx\dt\\
&\le\frac{\gm}{(1-\sig)^p r^p}\iint_{Q_{r}(\bar{\theta})}
		(u-k)^{p}_- \dx\dt\\
		&\quad+ \frac{\gm}{(1-\sig) \bar{\theta} r^p}\iint_{Q_{r}(\bar{\theta})}(u-k)_-^2\,\dx\dt\\
		&\quad+\frac{\gm }{(1-\sig) \bar{\theta} r^p}\iint_{Q_{r}(\bar{\theta})}(u-k)_-\,\dx\dt,
\end{align*}
provided that for some $i\in\{0,1,\cdots,\ell\}$,
\[
|\mu^- - e_i|\le\tfrac14\dl\xi\om\quad\text{ and }\quad \varep\le\tfrac14\dl\xi\om.
\]
\end{lemma}

Based on  the energy estimate in Lemma~\ref{Lm:energy}, a De Giorgi type lemma can be derived.
Notice that the time scaling used here is different from the one in Lemmas~\ref{Lm:DG:1} -- \ref{Lm:DG:2}.
\begin{lemma}\label{Lm:DG:3}
Suppose the hypotheses in Lemma~\ref{Lm:energy} hold.
Let $\dl\in(0,1)$. There exists a  constant $c_1\in(0,1)$ depending only on the data, such that
if
\[
\big|\big[u<\mu^-+2\dl\xi\om\big]\cap Q_{4\rho}(\widetilde{\theta})\big|\le c_1 (\xi\om)^{b}|Q_{4\rho}(\widetilde{\theta})|,\quad\text{ where }\widetilde{\theta}=(\dl\xi\om)^{1-p},
\]
then enforcing $|\mu^- - e_i |\le\tfrac14\dl\xi\om$ for some $i\in\{0,1,\cdots,\ell\}$ and $\varep\le\tfrac14\dl\xi\om$, we have
\[
u\ge\mu^-+\dl\xi\om\quad\text{ a.e. in }Q_{2\rho}(\widetilde{\theta}),
\]
provided $4^p\widetilde{\theta}\le A \theta$.
\end{lemma}
\begin{proof} 
For $n=0,1,\cdots,$ we set 
\begin{align*}
	\left\{
	\begin{array}{c}
	\displaystyle k_n=\mu^-+ \dl\xi\om+\frac{\dl\xi\om}{2^{n}},\quad \tilde{k}_n=\frac{k_n+k_{n+1}}2,\\[5pt]
	\displaystyle \varrho_n=2\varrho+\frac{\varrho}{2^{n-1}},
	\quad\tilde{\varrho}_n=\frac{\varrho_n+\varrho_{n+1}}2,\\[5pt]
	\displaystyle K_n=K_{\varrho_n},\quad \widetilde{K}_n=K_{\tilde{\varrho}_n},\\[5pt] 
	\displaystyle Q_n=Q_{\rho_n}(\widetilde{\theta}),\quad
	\widetilde{Q}_n=Q_{\tilde\rho_n}(\widetilde{\theta}).
	\end{array}
	\right.
\end{align*}
We will use the energy estimate in Lemma~\ref{Lm:energy} with the pair of cylinders $\widetilde{Q}_n\subset Q_n$.
Note that the constant $\bar\dl$ in Lemma~\ref{Lm:energy} is replaced by $(1+2^{-n})\dl$, as indicated in the definition of $k_n$.
Enforcing $|\mu^- - e_i|\le\tfrac14\dl\xi\om$ and $\varep\le\tfrac14\dl\xi\om$, the energy estimate in Lemma~\ref{Lm:energy} yields that
\begin{align*}
	(\dl\xi\om)^{-1}(\xi\om)^b&\essup_{-\widetilde{\theta}\tilde{\varrho}_n^p<t<0}
	\int_{\widetilde{K}_n} (u-\tilde{k}_n)_-^2\,\dx
	+
	\iint_{\widetilde{Q}_n}|D(u-\tilde{k}_n)_-|^p \,\dx\dt\\
	&\qquad\le
	 \gm \frac{2^{pn}}{\varrho^p}(\dl\xi\om)^{p}|A_n|,
\end{align*}
where
	$A_n:=\big[u<k_n\big]\cap Q_n$.

Let $0\le\phi\le1$ be a cutoff function that vanishes on the parabolic boundary of $\widetilde{Q}_n$
and equals the identity in $Q_{n+1}$. An application of the H\"older inequality, the Sobolev imbedding \cite[Chapter I, Proposition~3.1]{DB}
and the above energy estimate give that
\begin{align*}
	\frac{\dl\xi\om}{2^{n+3}}
	|A_{n+1}|
	&\le 
	\iint_{\widetilde{Q}_n}\big(u-\tilde{k}_n\big)_-\phi\,\dx\dt\\
	&\le
	\bigg[\iint_{\widetilde{Q}_n}\big[\big(u-\tilde{k}_n\big)_-\phi\big]^{p\frac{N+2}{N}}
	\,\dx\dt\bigg]^{\frac{N}{p(N+2)}}|A_n|^{1-\frac{N}{p(N+2)}}\\
	&\le \gm
	\bigg[\iint_{\tilde{Q}_n}\big|D\big[(u-\tilde{k}_n)_-\phi\big]\big|^p\,
	\dx\dt\bigg]^{\frac{N}{p(N+2)}}\\
	&\quad\ 
	\times\bigg[\essup_{-\widetilde{\theta}\tilde{\varrho}_n^p<t<0}
	\int_{\widetilde{K}_n}\big(u-\tilde{k}_n\big)^{2}_-\,\dx\bigg]^{\frac{1}{N+2}}
	 |A_n|^{1-\frac{N}{p(N+2)}}\\
	&\le 
	\gm (\xi\om)^{-\frac{b}{N+2}}(\dl\xi\om)^{\frac1{N+2}}
	\bigg[(\dl\xi\om)^{p}\frac{2^{pn}}{\rho^p}\bigg]^{\frac{N+p}{p(N+2)}}
	|A_n|^{1+\frac{1}{N+2}}. 
\end{align*}
In terms of $  Y_n=|A_n|/|Q_n|$, the recurrence is rephrased as
\begin{equation*}
	 Y_{n+1}
	\le
	\frac{\gm C^n}{(\xi\om)^{\frac{b}{N+2}}} 
	Y_n^{1+\frac{1}{N+2}},
\end{equation*}
for a constant $\gm$ depending only on the data and with $C=C(N,p)$.
Hence, by \cite[Chapter I, Lemma~4.1]{DB}, 
there exists
a positive constant $c_1$ depending only on the data, such that
$Y_n\to0$ if we require that $Y_o\le c_1(\xi\om)^b$.
This concludes the proof. 
\end{proof}

The next lemma concerns the smallness of the measure density of the set $[u\approx\mu^-]$.
Its proof relies on \eqref{Eq:energy-bis} and the measure information \eqref{Eq:measure:0} will be employed.
\begin{lemma}\label{Lm:shrink:1}
Let $u$ be a  weak super-solution to \eqref{Eq:reg} with \eqref{Eq:1:2} in $E_T$, under the measure information \eqref{Eq:measure:0}.
There exists a positive constant $\gm$ depending only on the data, such that
for any $j_*\in\nn$ we have 
\begin{equation*}
	\Big|\Big[
	u\le\mu^-+\frac{\xi \om}{2^{j_*}}\Big]\cap Q_{4\rho}(\widetilde{\theta})\Big|
	\le\frac{\gm}{ \bar\al} j_*^{-\frac{p-1}p} |Q_{4\rho}(\widetilde{\theta})|,
	\quad\text{ where }\widetilde{\theta}=\Big(\frac{\xi\om}{2^{j_*}}\Big)^{1-p},
\end{equation*}
with $\bar\al$ as in \eqref{Eq:choice-beta}, provided $4^p\widetilde{\theta}\le A \theta$.
\end{lemma}
\begin{proof}
We employ the energy estimate \eqref{Eq:energy-bis} in $Q_{8\rho}(\widetilde{\theta})$
with a standard cutoff function $\z$ that vanishes on the parabolic boundary of $Q_{8\rho}(\widetilde{\theta})$
and equals the identity in $Q_{4\rho}(\widetilde{\theta})$, satisfying $|D\z|\le \gm/\rho$ and $|\pl_t\z|\le\gm/(\widetilde\theta\rho^p)$.
The levels are chosen to be
\begin{equation*}
	k_j:=\mu^-+\frac{\xi \om}{2^{j}},\quad j=0,1,\cdots, j_*.
\end{equation*}

Therefore, assuming $j_*$ has been chosen, 
and taking into account the definition of $ \widetilde{\theta}(j_*)$, the energy estimate \eqref{Eq:energy-bis} yields that
\begin{equation*}
\begin{aligned}
	\iint_{Q_{4\rho}(\widetilde{\theta})}|D(u-k_j)_-|^p\,\dx\dt&\le\frac{\gm}{\varrho^p}\bigg(\frac{\xi \om}{2^j}\bigg)^p\bigg[1+\frac1{\widetilde{\theta}}\bigg(\frac{\xi\om}{2^j}\bigg)^{1-p}\bigg] |A_{j,8\rho}|\\
	&\le\frac{\gm}{\varrho^p}\bigg(\frac{\xi \om}{2^j}\bigg)^p|A_{j,8\rho}|,
\end{aligned}
\end{equation*}
where $ A_{j,8\rho}:= \big[u<k_{j}\big]\cap Q_{8\varrho}(\widetilde{\theta})$.
 
Observing $\xi<\bar\xi$ from \eqref{Eq:choice-xi-bar} and \eqref{Eq:choice-xi}, we may derive
the measure theoretical information from \eqref{Eq:measure:0}:
\begin{equation*}
	\big|\big[u(\cdot, t)>\mu^-+\xi \om\big]\cap K_{4\varrho}\big|\ge\bar\al 8^{-N} |K_{4\varrho}|
	\qquad\mbox{for all $t\in\big(-\widetilde{\theta}(4\rho)^p,0\big]$.}
\end{equation*}
With this information at hand, we apply \cite[Chapter I, Lemma 2.2]{DB} slicewise 
to $u(\cdot,t)$ for 
$t\in\big(-\widetilde{\theta}(4\rho)^p,0\big]$,
 over the cube $K_{4\varrho}$,
for levels $k_{j+1}<k_{j}$, followed by an application of H\"older's inequality.
Indeed, we estimate
\begin{align*}
	&(k_j-k_{j+1})\big|\big[u(\cdot, t)<k_{j+1} \big]
	\cap K_{4\varrho}\big|
	\\
	&\qquad\le
	\frac{\gm \varrho^{N+1}}{\big|\big[u(\cdot, t)>k_{j}\big]\cap K_{4\varrho}\big|}	
	\int_{ [k_{j+1}<u(\cdot,t)<k_{j}] \cap K_{4\varrho}}|Du(\cdot,t)|\,\dx\\
	&\qquad\le
	\frac{\gm\varrho}{\bar\al}
	\bigg[\int_{ [k_{j+1}<u(\cdot,t)<k_{j}] \cap K_{4\varrho}}|Du(\cdot,t)|^p\,\dx\bigg]^{\frac1p}
	\big|\big[ k_{j+1}<u(\cdot,t)<k_{j}\big]\cap K_{4\varrho}\big|^{1-\frac1p}
	\\
	&\qquad=
	\frac{ \gm\varrho}{\bar\al}
	\bigg[\int_{K_{4\varrho}}|D(u-k_j)_-(\cdot,t)|^p\,\dx\bigg]^{\frac1p}
	\big[ |A_{j,4\rho}(t)|-|A_{j+1,4\rho}(t)|\big]^{1-\frac1p},
\end{align*}
where we have set $ A_{j,4\rho}(t):= \big[u(\cdot,t)<k_{j}\big]\cap K_{4\varrho}$.
We perform an integration in $\d t$ over the interval $\big(-\widetilde{\theta}(4\rho)^p,0\big]$  on both sides
 and apply H\"older's inequality.
Setting $A_{j,4\rho}:=[u<k_j]\cap  Q_{4\rho}(\widetilde{\theta})$, we arrive at
\begin{align*}
	\frac{\xi \om}{2^{j+1}}\big|A_{j+1,4\rho}\big|
	&\le
	\frac{\gm\varrho}{\bar\al}\bigg[\iint_{Q_{4\rho}(\widetilde{\theta})}|D(u-k_j)_-|^p\,\dx\dt\bigg]^\frac1p
	\big[|A_{j,4\rho}|-|A_{j+1,4\rho}|\big]^{1-\frac{1}p}\\
	&\le
	\frac{\gm}{\bar\al} \frac{\xi \om}{2^j}|A_{o,8\rho}|^{\frac1p}\big[|A_{j,4\rho}|-|A_{j+1,4\rho}|\big]^{1-\frac{1}p}.
\end{align*}
%
Now take the power $\frac{p}{p-1}$ on both sides of the above inequality to obtain
\[
	\big|A_{j+1,4\rho}\big|^{\frac{p}{p-1}}
	\le
	\left(\frac{\gm}{\bar\al}\right)^{\frac p{p-1}}|A_{o,8\rho}|^{\frac1{p-1}}\big[|A_{j,4\rho}|-|A_{j+1,4\rho}|\big].
\]
Add these inequalities from $0$ to $j_*-1$ to obtain
\[
	j_* |A_{j_*,4\rho}|^{\frac{p}{p-1}}\le\left(\frac{\gm}{\bar\al}\right)^{\frac p{p-1}}|A_{o,8\rho}|^{\frac1{p-1}}|A_{o,4\rho}|,
\]
from which we easily obtain
\[
	|A_{j_*,4\rho}|\le\frac{\gm}{\bar\al}{j_*^{-\frac{p-1}p}}|Q_{4\rho}(\widetilde\theta)|.
\]
%
%
The proof is completed.
\end{proof}

\subsection{Reduction of oscillation near the infimum: Part IV}\label{S:reduc-osc-3}  
Under the conditions 
\eqref{Eq:opposite:1} and \eqref{Eq:opposite:2},
we may reduce the oscillation in the following way. First of all,
let $\xi$ be determined in \eqref{Eq:choice-xi}.
Then we choose, according to Lemma~\ref{Lm:shrink:1}, the integer $j_*$ so large to satisfy that
\[
\frac{\gm}{ \bar\al j_*^{\frac{p-1}p}}\le c_1(\xi\om)^b,
\]
 where $c_1$ is the constant appearing in Lemma~\ref{Lm:DG:3}.
 According to \eqref{Eq:choice-beta} and \eqref{Eq:choice-xi}, the dependence of $j_*$ can be traced by
\begin{equation}\label{Eq:choice-j}
j_*=\om^{-q_o}\exp\big\{\gm(\text{data})\om^{-q_1}\big\},
\end{equation}
for some positive $\{q_o,q_1\}$ depending on the data.

Next, we can fix $2\dl=2^{-j_*}$ in Lemma~\ref{Lm:DG:3}.
Consequently, by the choice of $j_*$ in  \eqref{Eq:choice-j}, Lemma~\ref{Lm:DG:3} can be applied,
assuming that $|\mu^- - e_i|\le\tfrac14\dl\xi\om$ for some $i\in\{0,1,\cdots,\ell\}$ and $\varep\le\frac14\dl\xi\om$, and we arrive at
\[
u\ge\mu^-+ \dl\xi\om \quad\text{ a.e. in }Q_{2\rho}(\widetilde{\theta}),
\]
where we may trace, recalling \eqref{Eq:choice-xi},
\begin{equation}\label{Eq:choice-dl-xi}
\dl=\exp\Big\{-\om^{-q_o}\exp\big\{\gm\om^{-q_1}\big\}\Big\}\quad\text{ and }\quad
\xi= \exp\Big\{- \gm\om^{-q_2}\Big\}
\end{equation}
for some generic $\gm$ and some positive $\{q_o, q_1, q_2\}$ determined by  the data.
This would give us a reduction of oscillation
\begin{equation}\label{Eq:reduc-osc-2}
\essosc_{Q_{2\rho}(\widetilde{\theta})}u\le(1-\dl\xi )\om
\end{equation}
with the  above-defined $\dl$ and $\xi$. The choice of $A$ can be finally made from
$8^p\widetilde{\theta}\le A\theta$ as required in \eqref{Eq:A-control}, i.e.  $A\ge 2^{p+4}\om^{-1}(\dl\xi)^{1-p}$.
Thus we may choose
\begin{equation}\label{Eq:choice-A}
A(\om)=\exp\Big\{\exp\big\{\gm \om^{-q}\big\}\Big\}
\end{equation}
for some properly defined positive $\gm$ and $q$ depending only on the data.

To summarize the achievements in \S\S~\ref{S:1st-alt} -- \ref{S:reduc-osc-3},  taking 
the reverse of \eqref{Eq:start-cylinder},
 \eqref{Eq:reduc-osc-0},  \eqref{Eq:reduc-osc-1} and \eqref{Eq:reduc-osc-2} all into account,
 if   $|\mu^- - e_i|\le\tfrac14\dl\xi\om$ for some $i\in\{0,1,\cdots,\ell\}$ and $\varep\le\frac14\dl\xi\om$ hold true, then  for $\theta=(\tfrac14\om)^{2-p}$ we have that
\begin{equation}\label{Eq:reduc-osc-3}
\text{ either }\ \ \ \om\le  \gm\Big(\ln^{(2)}\frac1{\rho}\Big)^{-\frac1q}\quad
\text{ or }\ \  \essosc_{Q_{\frac14\rho}(\theta)}u\le \big(1-\eta(\om)\big)\om,
\end{equation}
where 
\[
\eta=\exp\left\{-\exp\Big\{\exp\big\{\gm \om^{-q}\big\}\Big\}\right\},
\]
for some properly defined positive $\gm$ and $q$ depending only on the data.
\subsection{Reduction of oscillation near the infimum: Part V}\label{S:reduc-osc-4} 
Let $\xi(\om)$ and $\dl(\om)$ be determined in \eqref{Eq:choice-dl-xi}.
The analysis throughout  \S\S~\ref{S:reduc-osc-1} -- \ref{S:reduc-osc-3} 
has been founded on the condition \eqref{Eq:mu-minus}.
We now examine the case
when \eqref{Eq:mu-minus} does not hold, namely,
\begin{equation}\label{Eq:mu-minus-opp}
|\mu^- - e_i|>\tfrac14\dl\xi\om\quad\text{ for all }i\in\{0,1,\cdots,\ell\}.
\end{equation}
Notice that the analysis in \S~\ref{S:2nd-alt} does not rely on the condition \eqref{Eq:mu-minus},
and thus the measure information \eqref{Eq:measure:0} derived there is still at our disposal.
In view of the dependences of $\dl$ and $\xi$ in \eqref{Eq:choice-dl-xi}
and that of $\bar\xi$ in \eqref{Eq:choice-xi-bar}, we may assume that $\dl\xi<\bar\xi$
and  that \eqref{Eq:measure:0} holds true with $\bar\xi$ replaced by $\dl\xi$.

Next,  for $\widetilde\xi\in(0,\tfrac18)$ we introduce the levels
$k=\mu^-+  \widetilde\xi\dl\xi\om$.  
According to \eqref{Eq:mu-minus-opp} and assuming that $\varep\le\frac14\dl\xi\om$,
the energy estimate \eqref{Eq:en-est-gen}$_-$ written in $Q_{\rho}(\vartheta)\subset Q_\rho(A\theta)$ for some $0<\vartheta<A\theta$ yields that 
\begin{equation*}
\begin{aligned}
	\essup_{-\vartheta\rho^p<t<0}&\int_{K_\rho \times\{t\}}	
	\z^p (u-k)_-^2\,\dx 
	+
	\iint_{Q_{\rho}(\vartheta)}\z^p|D(u-k)_-|^p\,\dx\dt\nonumber\\
	&\le
	\gm\iint_{Q_{\rho}(\vartheta)}
		\Big[
		(u-k)^{p}_-|D\z|^p + (u-k)_-^2|\partial_t\z^p|
		\Big]
		\,\dx\dt.
\end{aligned}
\end{equation*}
Given this energy estimate and the measure information \eqref{Eq:measure:0},
the theory of parabolic $p$-Laplacian in \cite{DB} applies; see also \cite[Appendix~A]{Liao-Stefan} for tracing the constants. 
\begin{lemma}\label{Lm:reduc-osc-2}
Let $u$ be a  weak super-solution to \eqref{Eq:reg} with \eqref{Eq:1:2} in $E_T$.
Suppose \eqref{Eq:measure:0} and \eqref{Eq:mu-minus-opp} hold true, and $\varep\le\frac14\dl\xi\om$. 
There exists a positive constant $\widetilde\xi$ depending on
the data and $\bar\al$ of \eqref{Eq:choice-beta}, 
such that for $ \vartheta=(\widetilde\xi\dl\xi\om)^{2-p}$ we have
\[
\essosc_{Q_{\frac14\rho}( \vartheta)}u\le (1-  \widetilde\xi\dl\xi)\om,
\]
provided $\vartheta\le A\theta$.
Moreover, the dependence of $\widetilde\xi$ can be traced by 
$$\widetilde\xi=\gm(\text{\rm data}) \exp\big\{-\bar\al^{-\frac{p}{p-1}}\big\}.$$
\end{lemma}
\begin{remark}\upshape
Note that the choice of $A$ in \eqref{Eq:choice-A} verifies $\vartheta\le A\theta$.
\end{remark}
\subsection{Derivation of the modulus of continuity}\label{S:modulus}
This is the final part of the proof of Theorem~\ref{Thm:1:1}. 
Let us summarize what has been achieved by the previous sections.
To do so, we will first assume that $\om\le1$.
According to \eqref{Eq:reduc-osc-3} and Lemma~\ref{Lm:reduc-osc-2},
we have 
\begin{equation}\label{Eq:reduc-osc-4}
\essosc_{Q_{\frac14\rho}(\theta)}u\le \big(1-\eta(\om)\big)\om\quad
\text{ or }\quad \om\le  \gm\Big(\ln^{(2)}\frac1{\rho}\Big)^{-\frac1q}\quad\text{ or }\quad 
\om\le  \gm\Big(\ln^{(2)}\frac1{\varep}\Big)^{-\frac1q},
\end{equation}
where $\theta=(\tfrac14\om)^{2-p}$ and 
\[
\eta=\exp\left\{-\exp\Big\{\exp\big\{\gm \om^{-q}\big\}\Big\}\right\},
\]
for some properly defined positive $\gm$ and $q$ depending only on the data.

In order to iterate the arguments, we set
\[
\om_1:=\max\Big\{\big(1-\eta(\om)\big)\om, \gm\Big(\ln^{(2)}\frac1{\rho}\Big)^{-\frac1q}\Big\}
\]
and seek $\rho_1$ to verify the set inclusions, recalling $A$ from \eqref{Eq:choice-A}:
\[
Q_{\rho_1}(A_1\theta_1)\subset Q_{\frac14\rho}(\theta),\quad Q_{\rho_1}(A_1\theta_1)\subset  Q_o,
\]
where $\theta_1:=(\tfrac14\om_1)^{2-p}$, $A_1:=A(\om_1)$. 
Note that we may assume  $\eta(\om)\le\frac12$, which yields $\om_1\ge\frac12\om$. Then we estimate
\[
A_1\theta_1\rho_1^p\le A_1(\tfrac18\om)^{2-p}\rho_1^p,
\]
and hence choose
\[
A_1(\tfrac18\om)^{2-p}\rho_1^p=\theta(\tfrac14\rho)^p,\quad\text{ i.e. }\quad\rho_1^p=2^{2-3p}A_1^{-1}\rho^p.
\]
It is not hard to verify that  the other set inclusion also holds with such a choice of $\rho_1$.
Consequently, we arrive at
\[
\essosc_{Q_{\rho_1}(A_1\theta_1)}u\le \om_1
\]
which takes the place of \eqref{Eq:start-cylinder}$_2$ in the next stage.
Repeating the arguments of \S\S~\ref{S:1st-alt} -- \ref{S:reduc-osc-4},
we obtain that 
\begin{equation*}
\essosc_{Q_{\frac14\rho_1}(\theta_1)}u\le \big(1-\eta(\om_1)\big)\om_1\quad
\text{ or }\quad \om_1\le  \gm\Big(\ln^{(2)}\frac1{\rho_1}\Big)^{-\frac1q}\quad\text{ or }\quad 
\om_1\le  \gm\Big(\ln^{(2)}\frac1{\varep}\Big)^{-\frac1q}.
\end{equation*}

Now we may   construct for $n\in\nn$,
\begin{equation*}
\left\{
\begin{array}{cc}
\rho_o=\rho,\quad  \rho_{n+1}^p=2^{2-3p} A_n^{-1}\rho_n^p,\quad A_n=A(\om_n),\,\quad\theta_n=(\tfrac14\om_n)^{2-p},\\[5pt]
\dsty\om_o=\om,\quad \om_{n+1}=\max\Big\{\om_n\big(1-\eta(\om_n)\big),\,\gm\Big(\ln^{(2)}\frac1{\rho_n}\Big)^{-\frac1q}\Big\},\\[5pt]
Q_n=Q_{\frac14\rho_n}(\theta_n),\quad
Q'_n=Q_{\rho_n}(A_n \theta_n).
\end{array}\right.
\end{equation*}
By induction, if up to some $j\in\nn$, we have 
\[
\om_{n}>  \gm\Big(\ln^{(2)}\frac1{\varep}\Big)^{-\frac1q} \quad\text{ for all }n\in\{0,1,\cdots, j-1\},
\]
then  for all $n\in\{0,1,\cdots, j\}$, there holds
\[
Q'_{n}\subset Q_{n-1},\quad\essosc_{Q_n} u\le \om_n.
\]
On the other hand, we denote  by $j$ the first index to satisfy
\begin{equation}\label{Eq:osc-esp}
\om_{j}\le  \gm\Big(\ln^{(2)}\frac1{\varep}\Big)^{-\frac1q}.  
\end{equation}

Observe that if there exist $n_o\in\nn$ and a sequence $\{a_n\}$ satisfying
\[
a_{n+1}\ge\max\Big\{a_n\big(1-\eta(a_n)\big),\,\gm\Big(\ln^{(2)}\frac1{\rho_n}\Big)^{-\frac1q}\Big\}
\]
for all $n\ge n_o$, and meanwhile $a_{n_o}\ge \om_{n_o}$, then $a_n\ge \om_n$ for all $n\ge n_o$.
We may choose
\[
a_n=\Big(\ln^{(3)} (n+a)\Big)^{-\sig}
\]
for some proper $\sig\in(0,\tfrac1q)$ and an absolute constant $a$, such that $a_o\ge1$.  Since we have assumed $\om\le1$, we have $a_o\ge \om_o$
and hence, $a_n\ge \om_n$ for all $n\ge0$.

Let us take $4r\in(0,\rho)$. If for some $n\in\{0,1,\cdots, j\}$, we have
\[
\rho_{n+1}\le 4r<\rho_{n},
\]
then the right-hand side inequality yields
\begin{equation}\label{Eq:osc-est}
\essosc_{Q_{r}(\om^{2-p})} u\le\essosc_{Q_n} u \le\om_n\le  \Big(\ln^{(3)} (n+a)\Big)^{-\sig}.
\end{equation}
Next we examine the left-hand side inequality. For this purpose, we first note that
it may be assumed that $\eta(\om_n)\le\frac12$. Hence, we estimate $\om_n\ge(\frac12)^n\om$,
\[
A(\om_n)\le\exp\Big\{\exp\Big\{\frac{2^{qn}}{\om^q}\Big\}\Big\},
\]
and 
\begin{align*}
(4r)^p\ge\rho_{n+1}^p&\ge 2^{2-3p}\exp\Big\{-\exp\Big\{\frac{2^{qn}}{\om^q}\Big\}\Big\}\rho_n^p\\
&\ge 2^{ n(2-3p)}\exp\Big\{-\sum_{i=0}^{n}\exp\Big\{\frac{2^{qi}}{\om^q}\Big\}\Big\}\rho^p\\
&\ge 2^{ n(2-3p)}\exp\Big\{-\exp\Big\{\frac{2^{(q+1)n}}{\om^{q+1}}\Big\}\Big\}\rho^p.
\end{align*}
By taking logarithm on both sides, we estimate
\[
n\ge\frac1{(q+1)\ln2}\ln^{(3)}\frac{\rho}{cr} + |\ln \om|,
\]
for some absolute constant $c>0$.
Substituting it back to \eqref{Eq:osc-est}, we obtain
\[
\essosc_{Q_{r}(\om^{2-p})} u\le C \Big(\ln^{(6)}\frac{\rho}{cr}\Big)^{-\sig},
\]
for  some $C>0$ depending on the data and $\om$.

Finally, if $4r<\rho_{j+1}$ where $j$ is the first index for  \eqref{Eq:osc-esp} to hold, then we may use  \eqref{Eq:osc-esp}
and
\[
\essosc_{Q_{r}(\om^{2-p})}u\le \essosc_{Q_{j}} u\le \om_{j},
\]
 to incorporate the $\varep$ term into the oscillation estimate:
 \[
\essosc_{Q_{r}(\om^{2-p})}u \le C \Big(\ln^{(6)}\frac{\rho}{cr}\Big)^{-\sig}+\gm\Big(\ln^{(2)}\frac1{\varep}\Big)^{-\frac1q}.
\]
Now according to our assumption in \S~\ref{S:1:3} we may let $\varep\to0$ and obtain the desired
modulus of continuity.

The assumption that $\om\le1$ at the beginning of this section is not restrictive.
For otherwise, the same arguments in the previous sections generate quantities
$$\big\{\eta, \dl, \xi,\bar\xi,\widetilde\xi, \al,\bar\al,\kappa, A\big\}$$ depending only on the data, but independent of $\om$.
Consequently, instead of \eqref{Eq:reduc-osc-4}, we end up with
\begin{equation*}
\essosc_{Q_{\frac14\rho}(\theta)}u\le (1-\eta )\om\quad
\text{ or }\quad \om\le  4 (A\rho )^{\frac1{p-2}}\quad\text{ or }\quad 
\om\le  \frac{4\varep}{\dl\xi}.
\end{equation*}
Given this, we may set up an iteration scheme as before and iterate $n_*=n_*(\text{data})$ times, such that 
\[
\essosc_{Q_{\rho_{*}}(A \theta_{*})} u\le \om_*<1\quad\text{ where }\theta_*=(\tfrac14\om_*)^{2-p},
\]
for some $\rho_*$ and $\om_*$ depending on $n_*$. 

Without loss of generality,  due to \eqref{Eq:choice-A}, we may take $A=A(\om_*)$.
As such the above intrinsic relation plays the role of \eqref{Eq:start-cylinder} and 
the previous arguments can be reproduced.

%
%
%


\end{document}

%% file: dibe_1009.tex
\newtheorem{proposition}{Proposition}[section]
\newtheorem{theorem}{Theorem}[section]
\newtheorem{lemma}{Lemma}[section]
\newtheorem{corollary}{Corollary}[section]
\newtheorem{remark}{Remark}[section]

\numberwithin{equation}{section}
\numberwithin{theorem}{section}
\numberwithin{proposition}{section}
\numberwithin{lemma}{section}
\numberwithin{remark}{section}
\newcommand{\noi}{\noindent}

\newcommand{\dsty}{\displaystyle}

\newcommand{\al}{\alpha}

\newcommand{\be}{\beta}

\newcommand{\Gm}{\Gamma}
\newcommand{\gm}{\gamma}
\newcommand{\dl}{\delta}
\newcommand{\Dl}{\Delta}
\newcommand{\lm}{\lambda}

\newcommand{\varep}{\varepsilon}

\newcommand{\sig}{\sigma}

\newcommand{\om}{\omega}

\newcommand{\z}{\zeta}

\newcommand{\nn}{\mathbb{N}}

\newcommand{\rr}{\mathbb{R}}
\newcommand{\rn}{\rr^N}

\newcommand{\bl}[1]{\mathbf{#1}}

\newcommand{\dvg}{\operatorname{div}}

\newcommand{\essup}{\operatornamewithlimits{ess\,sup}}
\newcommand{\essinf}{\operatornamewithlimits{ess\,inf}}
\newcommand{\essosc}{\operatornamewithlimits{ess\,osc}}

\newcommand{\loc}{\operatorname{loc}}

\newcommand{\dist}{\operatorname{dist}}

\newcommand{\pl}{\partial}

\newcommand{\intl}{\int\limits}

\def\Xint#1{\mathchoice
    {\XXint\displaystyle\textstyle{#1}}%
    {\XXint\textstyle\scriptstyle{#1}}%
    {\XXint\scriptstyle\scriptscriptstyle{#1}}%
    {\XXint\scriptscriptstyle\scriptscriptstyle{#1}}%
    \!\int}
\def\XXint#1#2#3{\setbox0=\hbox{$#1{#2#3}{\int}$}
    \vcenter{\hbox{$#2#3$}}\kern-0.5\wd0}
\def\bint{\Xint-}
\def\dashint{\Xint{\raise4pt\hbox to7pt{\hrulefill}}}
\def\dashiint{\bint\kern-0.15cm\bint}

\newcommand{\ovl}[3]{\int_{#1}^{#2}\kern-#3pt\raise4pt\hbox to7pt{\hrulefill}\ }

\newcommand{\ovll}[3]{\intl_{#1}^{#2}\kern-#3pt\raise4pt\hbox to7pt{\hrulefill}\ }

\newcommand{\tvl}[2]{\iint_{#1}\kern-#2pt\raise4pt\hbox to7pt{\hrulefill}\ }

\newcommand{\bye}{\end{document}}

%% file: multiphase_part1_18.bbl
\begin{thebibliography}{99}
\bibitem{Urbano-14} P. Baroni, T. Kuusi and J.M. Urbano, {\it A quantitative modulus of continuity for the two-phase Stefan problem},
 Arch. Ration. Mech. Anal., {\bf214}(2), (2014),  545--573.
\bibitem{BKLU-18} P. Baroni, T. Kuusi, C. Lindfors and J.M. Urbano, 
{\it Existence and boundary regularity for degenerate phase transitions},
SIAM J. Math. Anal., {\bf 50}(1), (2018), 456--490.
\bibitem{Caff-Evans-83} L.A. Caffarelli and L.C. Evans,  {\it Continuity of the temperature in the two-phase Stefan problem},
 Arch. Ration. Mech. Anal., {\bf81}(3), (1983),  199--220.
\bibitem{DB-82} E. DiBenedetto, {\it Continuity of weak solutions to certain singular parabolic equations},
 Ann. Mat. Pura Appl. (4), {\bf130}, (1982), 131--176.
\bibitem{DB-86} E. DiBenedetto, {\it A boundary modulus of continuity for a class of singular parabolic equations},
 J. Differential Equations, {\bf63}(3), (1986), 418--447. 
\bibitem{DB} E. DiBenedetto, ``Degenerate Parabolic 
Equations," Universitext, Springer-Verlag, New York, 1993.  
\bibitem{DiBe-Fried} E. DiBenedetto and A. Friedman, {\it Regularity of solutions of nonlinear 
degenerate parabolic systems}, J. Reine Angew. Math., 349, (1984), 83--128.
\bibitem{DiBe-Gar} E. DiBenedetto and R. Gariepy, 
{\it Local behavior of solutions of an elliptic-parabolic equation},
Arch. Ration. Mech. Anal., {\bf 97}(1), (1987), 1--17.
\bibitem{DBGV-mono} E. DiBenedetto, U. Gianazza and V. Vespri, 
``Harnack's Inequality for Degenerate and Singular Parabolic 
Equations," Springer Monographs in Mathematics, Springer-Verlag, 
New York, 2012.
\bibitem{DBUV} E. DiBenedetto, J.M. Urbano and V. Vespri,
{\it Current issues on singular and degenerate evolution equations},
``Handbook of differential equations: evolutionary equations," {V}ol. {I},
Handb. Differ. Equ., 169--286, North-Holland, Amsterdam, 2004.
\bibitem{DBV-95}  E. DiBenedetto  and V. Vespri, {\it On the singular equation $\be(u)_t=\Dl u$},
 Arch. Ration. Mech. Anal., {\bf132}(3), (1995), 247--309. 
\bibitem{GV-03} U. Gianazza and V. Vespri, {\it Continuity of weak solutions of a singular parabolic equation},
 Adv. Differential Equations, {\bf8}(11), (2003),  1341--1376. 
\bibitem{KrS} N.V. Kru\v{z}kov and S.M. Sukorjanski\u{\i}, 
{\it Boundary Value Problems for Systems of Equations of two
Phase Porous Flow Type: Statement of the Problems, Questions of
Solvability, Justification and Approximation Methods}, Math. Sbornik {\bf44}, (1977), 62--80.
\bibitem{LSU} O.A. {Lady\v{z}enskaja}, V.A. Solonnikov and N.N. Ural'ceva, 
``Linear and quasilinear equations of parabolic type," Translations of Mathematical Monographs, Vol. 23,
American Mathematical Society, Providence, R.I., 1968.
\bibitem{Le} M.C. Leverett, 
{\it Capillary Behaviour in Porous Solids}, Trans. Amer. Inst.
Mining and Metallurgicals Engrs. {\bf 142}, (1941), 151--169.
\bibitem{Liao-Stefan} N. Liao, {\it On the logarithmic type boundary modulus of continuity for  the Stefan problem}, 
arXiv: 2102.10278.
\bibitem{Liao-improved} N. Liao, {\it An improved modulus of continuity for the two-phase Stefan problem}, preprint.
\bibitem{Olienik} O.A. {Ole\u{\i}nik}, {\it A method of solution of the general {S}tefan problem},
Soviet Math. Dokl., {\bf 135}(5), (1960), 1350--1354.
\bibitem{Sacks} P.E. Sacks, {\it Continuity of solutions of a singular parabolic equation}, Nonlinear Anal., {\bf 7}(4),(1983), 387--409.
\bibitem{Urbano-00} J.M. Urbano, {\it Continuous solutions for a degenerate free boundary problem},
 Ann. Mat. Pura Appl. (4), {\bf178}, (2000), 195--224.
\bibitem{Visintin} A. Visintin, {\it Introduction to {S}tefan-type problems}, ``Handbook of differential equations: evolutionary equations," {V}ol. {IV}, Handb. Differ. Equ., 377--484, Elsevier/North-Holland, Amsterdam, 2008.
\bibitem{Ziemer} W.P. Ziemer, {\it Interior and boundary continuity of weak solutions of degenerate parabolic equations}, Trans. Amer. Math. Soc., {\bf 271}(2), (1982), 733--748.
\end{thebibliography}
